\documentclass[a4paper,10pt]{amsart}
\usepackage[cp1250]{inputenc}
\usepackage[T1]{fontenc}
\usepackage{amsfonts}
\usepackage{amsmath}
\usepackage{amssymb}
\usepackage{amsthm}
\usepackage{textcomp}

\usepackage{graphicx}
\usepackage[all]{xy}

\usepackage{placeins}
\usepackage{fancyhdr}
\usepackage{enumerate}

\newtheorem{dfn}{Definition}[section]

\newtheorem{prop}[dfn]{Proposition}
\newtheorem{thm}[dfn]{Theorem}

\newtheorem{cor}[dfn]{Corollary}

\theoremstyle{definition}
\newtheorem{rem}[dfn]{Remark}

\begin{document}
\title{DG categories and exceptional collections}

\author[A.~Bodzenta]{Agnieszka Bodzenta}
\address{Faculty of Mathematics, Informatics and Mechanics,
         University of Warsaw,
	 Banacha 2,
	 02-097 Warsaw, Poland}
\email{a.bodzenta@mimuw.edu.pl}


\begin{abstract}
In \cite{bib_BK} Bondal and Kapranov describe how to assign to a full exceptional collection on a variety $X$ a DG category $\mathcal{C}$ such that the bounded derived category of coherent sheaves on $X$ is equivalent to the bounded derived category of $\mathcal{C}$. In this paper we show that the category $\mathcal{C}$ has finite dimensional spaces of morphisms. We describe how it behaves under mutations and present an algorithm allowing to calculate it for full exceptional collections with vanishing Ext$^k$ groups for $k>1$. Finally, we use it to describe an example of a non-commutative deformation of certain rational surfaces.
\end{abstract}

\let\thefootnote\relax\footnote{The author is supported by a Polish MNiSW grant (contract number N N201 420639). }

\maketitle 

\section*{Introduction}

Derived categories of coherent sheaves have been an object of vivid interest since their discovery in 1960's. Full strong exceptional collections are one of the most important tools used to study them. As proved in \cite{bib_B}, such a collection on a smooth projective variety $X$ provides an elegant description of the category $D^b(\textrm{Coh}(X))$ -- it is equivalent to a derived category of modules over a quiver with relations. If an exceptional collection $\sigma$ is full but not strong the description becomes more complicated. In \cite{bib_BK} Bondal and Kapranov prove that in this case the category $D^b(\textrm{Coh}(X))$ is equivalent to a bounded derived category of modules over some DG category $\mathcal{C}_\sigma$. It follows from the construction that $\mathcal{C}_\sigma$ has finitely many objects.

The DG category $\mathcal{C}_\sigma$ is a full subcategory of an enhancement of $D^b(\textrm{Coh}(X))$ -- a DG category $\widetilde{\mathcal{C}}$ such that $H^0(\widetilde{\mathcal{C}})$ is equivalent to $D^b(\textrm{Coh}(X))$. Lunts and Orlov in \cite{bib_LO} show that this enhancement is strongly unique. Canonically the DG category $\widetilde{\mathcal{C}}$ is a subcategory of complexes of injective sheaves. Such a description suggests that the category $\mathcal{C}_\sigma$ has infinitely dimensional spaces of morphisms between objects and therefore it is inconvenient to work with.

In this paper, using the theory of $A_\infty$ categories, we show that $\mathcal{C}_\sigma$ can be in fact chosen to be finite and oriented. It means that one can partially order objects of $\mathcal{C}_\sigma$ in such a way that there are no nontrivial morphisms from $C_1$ to $C_2$ if $C_2 \preceq C_1$. Furthermore, all spaces of morphisms in $\mathcal{C}_\sigma$ are finite dimensional. Thus, $\mathcal{C}_\sigma$ can be graphically presented as a DG quiver -- a quiver with grading and differential on arrows.

The set of full exceptional collections on $X$ admits an action of the braid group. It is given by mutations defined in \cite{bib_B}. Using twisted complexes we lift this action to associated DG categories. It allows us to describe the category $\mathcal{C}_\sigma$ for any collection $\sigma$ which can be mutated to a strong one.

Families of full exceptional collections on rational surfaces have been given in \cite{bib_HP}. Hille and Perling also describe how to obtain a tilting object from a full exceptional collection with Ext$^k$ groups vanishing for $k>1$. Using their idea of universal extensions we can calculate the DG category $\mathcal{C}_\sigma$ for any full exceptional collection $\sigma = \langle \mathcal{E}_1,\ldots, \mathcal{E}_n \rangle$ with $\textrm{Ext}^k(\mathcal{E}_i, \mathcal{E}_j) = 0$ for $k>1$. In particular, this condition is satisfied by full exceptional collections of line bundles on  smooth rational surfaces.

We calculate the category $\mathcal{C}_\sigma$ for an exceptional collection on $\mathbb{P}^2$ blown up along a subscheme of degree two supported at one point. It turns out that finding a triple Massey product of morphisms in this collection is the most important part of the whole calculation. 

Full strong exceptional collections have also been used in \cite{bib_AKO}, \cite{bib_HI} and \cite{bib_P} to describe non-commutative deformations of varieties. Knowledge of the DG categories of not necessarily strong exceptional collections can lead to a generalisation of this idea. We present an example of a non-commutative deformation from $\mathbb{P}^2$ blown up in two different points to $\mathbb{P}^2$ blown up along a subscheme of degree 2 supported at one point.

 \section{Enhanced triangulated categories}
Let us recall definitions and facts about enhanced triangulated categories and exceptional collections after \cite{bib_B}, \cite{bib_BK}, \cite{bib_BLL} and \cite{bib_K3}.

\subsection*{Exceptional collections and mutations}

Let $\mathcal{D}$ be a $\mathbb{C}$-linear triangulated category. An object $E\in \mathcal{D}$ is called \emph{exceptional} if $\mathcal{D}(E,E) = \mathbb{C}$ and $\mathcal{D}(E,E[i]) = 0$ for $i\neq 0$. An ordered collection \mbox{$\sigma = \langle E_1, \ldots, E_n\rangle$} of exceptional objects is called an \emph{exceptional collection} if \mbox{$\mathcal{D}(E_i, E_j[k]) = 0$} for $i>j$ and all $k$. An exceptional collection $\sigma$ is \emph{strong} if we also have \mbox{$\mathcal{D}(E_i, E_j[k]) = 0$} for $k\neq 0$ and all $i$, $j$. It is \emph{full} if the smallest strictly full subcategory of $\mathcal{D}$ containing $E_1,\ldots, E_n$ is equal to $\mathcal{D}$.
 
Let $ \langle E, F\rangle $ be an exceptional pair in $\mathcal{D}$. Then $ \langle L_EF,E\rangle $ and $ \langle F, R_FE\rangle $ are also exceptional pairs for $L_EF$ and $R_FE$ defined by means of distinguished triangles in $\mathcal{D}$.

\begin{align*}
L_EF \to \mathcal{H}om_\mathcal{D}(E,F) \otimes E \to F \to L_EF[1],\\
E \to \mathcal{H}om_\mathcal{D}(E,F)^*\otimes F \to R_FE\to E[1].
\end{align*}

Here, $\mathcal{H}om_\mathcal{D}(E,F)$ denotes a complex of $\mathbb{C}$-vector spaces with trivial differential; $\mathcal{H}om_\mathcal{D}(E,F) = \bigoplus \mathcal{D}(E,F[k])$. For an element $E \in \mathcal{D}$ and a complex $V^\bullet$ the tensor product is defined by \mbox{$E\otimes V^\bullet = \bigoplus_{k\in \mathbb{Z}} \bigoplus_{i=0}^{\textrm{dim}V^k}E[-k]$}.

For an exceptional collection $\sigma = \langle E_1,\ldots,E_n\rangle$ the $i$-th left mutation $L_i\sigma$ and the $i$-th right \emph{mutation} $R_i\sigma$ are defined by 
\begin{align*}
L_i\sigma = \langle E_1, \ldots, E_{i-1}, L_{E_i}E_{i+1}, E_i, E_{i+2},\ldots, E_n  \rangle,\\
R_i\sigma = \langle E_1,\ldots, E_{i-1}, E_{i+1}, R_{E_{i+1}}E_i, E_{i+2},\ldots, E_n \rangle.
\end{align*}

Mutations define an action of the braid group on the set of full exceptional collections. 

\subsection*{Enhanced triangulated categories}

\begin{dfn}
A \emph{DG category} is a preadditive category $\mathcal{C}$ in which abelian groups $\mathcal{C}(A,B)$ are endowed with a $\mathbb{Z}$-grading and a differential $\partial$ of degree one. The composition of morphisms 
$$
\mathcal{C}(A,B)\otimes \mathcal{C}(B,C)\to \mathcal{C}(A,C)
$$
is a morphism of complexes and for any object $C\in\mathcal{C}$ the identity morphism $\textrm{id}_C$ is a closed morphism of degree zero.
\end{dfn}

We assume that all DG categories are $\mathbb{C}$ - linear.

For an element $x$ in a graded vector space we will denote by $|x|$ the grading of $x$. By $\mathcal{C}^i(A,B)$ we will denote morphisms of degree $i$.

A DG category $\mathcal{C}$ is \emph{ordered} if there exists a partial order $\preceq $ on the set of objects such that $\mathcal{C}(A,B) =0$ for $B\preceq A$. It is \emph{finite} if the set of objects of $\mathcal{C}$ is finite and for any $C_1, C_2 \in \mathcal{C}$ the vector space $\mathcal{C}(C_1,C_2)$ is finite dimensional.

To a DG category $\mathcal{C}$ we associate three categories. The graded category  $\mathcal{C}^\textrm{gr}$ is obtained from $\mathcal{C}$ by forgetting the differentials on morphisms while the \emph{homotopy category} $H(\mathcal{C})$ has the same objects as $\mathcal{C}$ and morphisms given by the cohomology groups of morphisms in $\mathcal{C}$. A further restriction to the zeroth cohomology gives a preadditive category $H^0(\mathcal{C})$.

A morphism $s:C\to C'$ in $\mathcal{C}$ is a \emph{homotopy equivalence} if $H(s)$ is an isomorphism. Then we say that $C$ and $C'$ are \emph{homotopy equivalent}.

A \emph{DG functor} between two DG categories $\mathcal{C}$ and $\mathcal{C}'$ is an additive functor \mbox{$F:\mathcal{C}\to \mathcal{C}'$} which preserves the grading and differential on morphisms. DG functors between DG categories form a DG category. 

Let \mbox{$F, G: \mathcal{C}_1 \to \mathcal{C}_2$} be DG functors. To construct a DG category \mbox{DG-Fun$(\mathcal{C}_1,\mathcal{C}_2)$} we put \mbox{$\textrm{DG-Fun}(\mathcal{C}_1,\mathcal{C}_2)^k(F,G)$} to be the set of natural transformations \mbox{$t:F^\textrm{gr} \to G^\textrm{gr}[k]$} (i.e. for $C\in \mathcal{C}_1$ we have $t_C\in \mathcal{C}_2^k(F(C),G(C))$).  The differential $\partial$ is defined pointwise; for $t_C:F(C) \to G[k](C)$ we have \mbox{$(\partial(t))_C = \partial (t_C):F(C) \to G[k-1](C)$}. 

The category of contravariant DG functors is denoted by DG-Fun$^{\textrm{o}}(\mathcal{C},\mathcal{C}')$. A functor $F:\mathcal{C} \to \mathcal{C}'$ is a \emph{quasi-isomorphism} if $H(F):H(\mathcal{C}) \to H(\mathcal{C}')$ is an isomorphism.

The category $\textrm{DGVect}_\mathbb{C}$ of complexes of vector spaces with homogenous morphisms $f:V^\bullet \to W^\bullet$, $f(V^i) \subset W^{i+k}$ and a differential 
$$
\partial(f)^i = \partial_W f^{i+1} - (-1)^{|f|} f^{i+1}\partial_V
$$
is a DG category. 

A \emph{right DG module} $M$ over a DG category $C$ is an element of $\textrm{DG-Fun}^{\textrm{o}}(\mathcal{C},\textrm{DGVect}_\mathbb{C})$. After \cite{bib_K3} we define a derived category $D(\mathcal{C})$ as a localization of $H^0(\textrm{DG-Fun}^{\textrm{o}}(\mathcal{C},\textrm{DGVect}_\mathbb{C}))$ with respect to the class of homotopy equivalences. By $D^b(\mathcal{C})$ we will denote the subcategory of $D(\mathcal{C})$ formed by compact objects. The Yoneda embedding gives a functor \mbox{$h: \mathcal{C} \to D^b(\mathcal{C})$} which assigns to every $C\in \mathcal{C}$ a free module \mbox{$h_C = \mathcal{C}(-,C)$}.

For a DG category $\mathcal{C}$ we define the category $\widehat{\mathcal{C}}$ of formal shifts. The objects of $\widehat{\mathcal{C}}$ are of the form $C[n]$ where $C\in \mathcal{C}$ and $n\in\mathbb{N}$. For elements $C_1[k]$ and $C_2[n]$ of $\widehat{\mathcal{C}}$ we put $\textrm{Hom}^l_{\widehat{\mathcal{C}}}
(C_1[k], C_2[n]) = \textrm{Hom}^{l+n-k}_\mathcal{C}(C_1, C_2)$. For appropriate sign convention see \cite{bib_BLL}.

Let $B, C$  be objects of a DG category $\mathcal{C}$ and let \mbox{$f\in \mathcal{C}(B, C)$} a closed morphism. Assume that $B[1]$ is also an object of $\mathcal{C}$; i.e. there exists an object $B'$ and closed morphisms $t:B\rightarrow B'$, $t':B'\rightarrow B$ of degree 1 and -1 respectively such that $t't = \textrm{id}_B$ and $tt' = \textrm{id}_{B'}$. An object $D\in \mathcal{C}$ is called a \emph{cone} of $f$ if there exist morphisms
\[
\xymatrix{B' \ar[r]^i & D \ar[r]^p & B', &  & C \ar[r]^j & D \ar[r]^s & C,}
\]
such that 
$$
pi=1, \quad sj=1, \quad si=0, \quad pj=0, \quad ip+js =1.
$$
It is proved in \cite{bib_BLL} that the cone of closed degree zero morphism is 
uniquely defined up to a DG isomorphism.

One can formally add cones of closed morphisms to a DG category $\mathcal{C}$ by considering the category $\mathcal{C}^\textrm{pre-tr}$ of one-sided twisted complexes over $\mathcal{C}$.

\begin{dfn}
A \emph{one-sided twisted complex} over a DG category $\mathcal{C}$ is an expression $(\bigoplus_{i=1}^n C_i[r_i], q_{i,j})$ where $C_i$'s are objects of $\mathcal{C}$, $r_i\in \mathbb{Z}$, $n\geq 0$ and $q_{i,j}\in\mathcal{C}^1(C_i[r_i], C_j[r_j])$ such that $q_{i,j}=0$ for $i\geq j$ and $\partial q + q^2 = 0$.

One-sided twisted complexes over $\mathcal{C}$ form a DG category $\mathcal{C}^\textrm{pre-tr}$ with the morphism space between \mbox{$C=(\bigoplus C_i[r_i], q)$} and $C'=(\bigoplus C'_j[r'_j], q')$ given by the set of matrices $f=(f_{i,j})$, $f_{i,j}: C_i[r_i] \to C'_j[r'_j]$. With a differential defined as
$$
\partial (f) = (\partial f_{i,j}) + q'f - (-1)^{\textrm{deg}f}fq
$$
the category $\mathcal{C}^\textrm{pre-tr}$ is a DG category. We will denote its zeroth homotopy category $H^0(\mathcal{C}^\textrm{pre-tr})$ by $\mathcal{C}^\textrm{tr}$.
\end{dfn}

Let $C=(\bigoplus C_i[r_i], q_{i,j})$ be an object of $\mathcal{C}^\textrm{pre-tr}$. The shift of $C$ is defined as $C[1]=(\bigoplus C_i[r_i+1], - q_{i,j})$. The category $\mathcal{C}^\textrm{pre-tr}$ is closed under formal shifts.
 
Let $f:C\to D$ be a closed morphism of degree 0 in $\mathcal{C}^\textrm{pre-tr}$, where \mbox{$C=(\bigoplus_{i=1}^n C_i[r_i], q_{i,j})$,} $D=(\bigoplus_{i=1}^mD_i[s_i],p_{i,j})$. \emph{The cone} of $f$ is a twisted complex $C(f)=(\bigoplus_{i=1}^{n+m}E_i[t_i], u_{i,j})$, where
\begin{align*}
E_i & = \left\{ \begin{array}{ll} C_i & \textrm{for } i\leq n,\\
D_{i-n} &\textrm{for } i>n,\end{array}  \right.\\
t_i & = \left\{ \begin{array}{ll} r_i + 1 & \textrm{for } i\leq n,\\
s_i & \textrm{for } i>n, \end{array}   \right.\\
u_{i,j} & = \left\{ \begin{array}{ll} q_{i,j} & \textrm{for } i,j \leq n, \\
f_{i, j-n} & \textrm{for } i\leq n < j,\\
p_{i-n,j-n} & \textrm{for } i,j > n. \end{array}  \right.
\end{align*}

The \emph{convolution} functor $\textrm{Tot} : (\mathcal{C}^\textrm{pre-tr})^\textrm{pre-tr} \to \mathcal{C}^\textrm{pre-tr}$ establishes a quasi-isomorphism between $(\mathcal{C}^\textrm{pre-tr})^\textrm{pre-tr}$ and $\mathcal{C}^\textrm{pre-tr}$. Let $C_i = (\bigoplus_{j=1}^{n_i} D^i_j[r^i_j],q^i_{jk})$ be objects of $\mathcal{C}^\textrm{pre-tr}$ and let $C=(\bigoplus_{i=1}^n C_i[r_i], q_{ij})$ be a twisted complex in  $(\mathcal{C}^\textrm{pre-tr})^\textrm{pre-tr}$. Then the convolution Tot$(C)$ is equal to \mbox{$(\bigoplus_{i=1}^n\bigoplus_{j=1}^{n^i_j}D^i_j[r^i_j +r_i],q^i_{jk}+q_{ij})$}.

The DG category $\mathcal{C}$ is \emph{pretriangulated} if the embedding $H^0(\mathcal{C}) \to \mathcal{C}^{tr}$ is an equivalence. The category $H^0(\mathcal{C})$ for a pretriangulated category is triangulated. 

A triangulated category $\mathcal{D}$ is \emph{enhanced} if it has an \emph{enhancement} -- a pretriangulated category $\mathcal{C}$ such that $\mathcal{D}$ is equivalent to $H^0(\mathcal{C})$.

Recall, that an additive category is Karoubian if every projector splits. The category $\mathcal{C}^\textrm{tr}$ needs not to be Karoubian. As showed in \cite{bib_BLL} the category $D^b(\mathcal{C})$ is the Karoubization of $\mathcal{C}^\textrm{tr}$.

A standard example of an enhanced triangulated category is the derived category $D(\mathcal{A})$ of an abelian category $\mathcal{A}$ with enough injectives $I$. Its enhancement is the category of complexes of injective sheaves $\textrm{Kom}(I)$.

With the above definitions we are ready to state the following theorem. 

\begin{thm}\label{thm:generating}\emph{[Theorem 1 of \cite{bib_BK}]}
Let $\widetilde{\mathcal{C}}$ be a pretriangulated category, $E_1,\ldots,E_n$ objects of $\widetilde{\mathcal{C}}$ and $\mathcal{C} \subset \widetilde{\mathcal{C}}$ the full DG subcategory on the objects $E_i$. Then the smallest triangulated subcategory of $H^0(\widetilde{\mathcal{C}})$ containing $ E_1,\ldots,E_n$ is equivalent to $\mathcal{C}^\textrm{tr}$ as a triangulated category.
\end{thm}

It follows that a full exceptional collection $\sigma = \langle E_1, \ldots, E_n\rangle$ in an enhanced triangulated category $\mathcal{D}$ leads to an equivalence of $\mathcal{D}$ and $\mathcal{C}_\sigma^\textrm{tr}$ for some DG category $\mathcal{C}_\sigma$. The category $\mathcal{C}_\sigma$ is a subcategory of the enhancement of $\mathcal{D}$ and thus $H^i(\mathcal{C}_\sigma(E,F)) = \mathcal{D}(E,F[i])$ for any elements $E$, $F$ of $\sigma$.

\begin{rem} If an enhanced triangulated category $\mathcal{D}$ is Karoubian, a full exceptional collection $\sigma$ leads to an equivalence of $\mathcal{D}$ and $D^b(\mathcal{C}_\sigma)$.
\end{rem}

\section{Properties of DG categories associated to full exceptional collections}

From now on let $X$ be a smooth projective variety. Let $D^b(X) = D^b(\textrm{Coh}(X))$ denote the bounded derived category of the category of coherent sheaves on $X$. It is an enhanced Karoubian triangulated category -- its enhancement is given by $\underline{Kom}^b(I)$ -- the DG category of complexes of injective sheaves bounded from below and with finitely many nonzero coherent cohomology groups.

Let $\sigma = \langle \mathcal{E}_1, \ldots, \mathcal{E}_n \rangle$ be a full exceptional collection of coherent sheaves on $X$. By Theorem \ref{thm:generating} there exists a DG category $\mathcal{C}_\sigma$ with objects $\mathcal{E}_1, \ldots, \mathcal{E}_n$ such that $D^b(X)$ is equivalent to $\mathcal{C}_\sigma^{\textrm{tr}}$. As $D^b(X)$ is Karoubian the same is true about $\mathcal{C}_\sigma^\textrm{tr}$ and hence $\mathcal{C}_\sigma^\textrm{tr} = D^b(\mathcal{C}_\sigma)$. As $\sigma$ is an exceptional collection the category $H(\mathcal{C}_\sigma)$ is ordered and finite.

As two quasi-isomorphic DG categories have equivalent derived categories, the category $\mathcal{C}_\sigma$ is determined up to a quasi-isomorphism.

\begin{rem}
If a DG category $\mathcal{C}$ has only zeroth cohomology then it is quasi-isomorphic to $H(\mathcal{C}) = H^0(\mathcal{C})$. Indeed,
for $C_1,C_2\in \mathcal{C}$ let $\mathcal{C}(C_1,C_2) = \bigoplus_{n\in \mathbb{Z}} \mathcal{C}^n (C_1,C_2)$ with differential 
$$
\partial_{C_1,C_2}^n:\mathcal{C}^n(C_1,C_2) \to \mathcal{C}^{n+1}(C_1,C_2)
$$ 
and let $\mathcal{C}_I$ be a DG category such that $\textrm{ob } \mathcal{C}_I = \textrm{ob } \mathcal{C}$ and 
$$
\mathcal{C}_I(C_1, C_2) = \bigoplus_{n<0} \mathcal{C}^n(C_1,C_2) \oplus \textrm{ker}\partial_{C_1,C_2}^0.
$$
Then the natural inclusion functor $\mathcal{C}_I\to \mathcal{C}$ is a quasi-isomorphism. Let us also set 
$$
J_{C_1,C_2} = \textrm{Im}(\partial_{C_1,C_2}^{-1}) \oplus (\bigoplus_{n<0} \mathcal{C}^n(C_1,C_2))
$$
for any $C_1$ and $C_2$ in $\mathcal{C}$ and consider the category $\mathcal{C}_{I/J}$ with $\textrm{ob }\mathcal{C}_{I/J} = \textrm{ob }\mathcal{C}$ and 
$$
\mathcal{C}_{I/J}(C_1,C_2) = \mathcal{C}_I(C_1,C_2)/J_{C_1,C_2}.
$$
Then $\mathcal{C}_{I/J}$ is isomorphic to $H(\mathcal{C})$ and the natural functor $\mathcal{C}_I \to \mathcal{C}_{I/J}$ is a quasi-isomorphism.

It shows that if $\sigma$ is a strong exceptional collection the DG category $\mathcal{C}_\sigma$ is quasi-isomorphic to an ordinary category.
\end{rem}

\subsection*{Finiteness}

As the enhancement of $\mathcal{D}$ is $\underline{\textrm{Kom}}^b(I)$, calculating the category $\mathcal{C}_\sigma$ requires taking injective resoultions. This suggest that the category $\mathcal{C}_\sigma$ can have infinitely dimensional morphisms spaces. However, it is $A_\infty$-quasi-isomorphic to an ordered DG category with finite-dimensional morphisms between objects.

\subsubsection*{$A_\infty$ categories}
We recall definitions after \cite{bib_K1} and \cite{bib_L}.

\begin{dfn}
An \emph{$A_{\infty}$ category} $\mathcal{A}$ over $\mathbb{C}$ consists of
\begin{itemize}
\item objects $\textrm{ob}(\mathcal{A})$,
\item for any two $A, B \in \textrm{ob}(\mathcal{A})$ a $\mathbb{Z}$ - graded $\mathbb{C}$-vector space $\mathcal{A}(A,B)$,
\item for any $n\geq 1$ and a sequence $A_0,A_1, \ldots, A_n \in \textrm{ob}(\mathcal{A})$ a graded map:
$$
m_n:\mathcal{A}(A_{n-1},A_{n})\otimes\ldots\otimes\mathcal{A}(A_0,A_1) \to \mathcal{A}(A_0, A_n)
$$
of degree $2-n$ such that for any $n$
$$
\sum_{r+s+t = n} (-1)^{r+st} m_{r+1+t}(\textrm{id}^{\otimes r}\otimes m_s \otimes \textrm{id}^{\otimes t}) = 0.
$$
\end{itemize}
\end{dfn}
 
Note that when these formulae are applied to elements additional signs appear because of the Koszul sign rule:
$$
(f\otimes g)(x\otimes y) = (-1)^{|x||g|}f(x)\otimes g(y).
$$

An $A_{\infty}$ \emph{algebra} is an $A_{\infty}$ category with one object.

An $A_\infty$ category $\mathcal{A}$ is \emph{ordered} if there exists a partial order $\preceq$ on the set $\textrm{ob}(\mathcal{A})$ such that $\mathcal{A}(A,A') = 0$ for $A'\preceq A$. It is \emph{finite} if ob$(\mathcal{A})$ is a finite set and $\mathcal{A}(A,A')$ is finite-dimensional for any $A$ and $A'$.

The operation $m_1$ gives for any $A, B \in \textrm{ob}(\mathcal{A})$ a structure of a complex on $\mathcal{A}(A,B)$. The homotopy category with respect to $m_1$ is denoted by $H(\mathcal{A})$.

The $A_\infty$ category is \emph{minimal} if the operation $m_1$ is trivial.

Any DG category can be regarded as an $A_{\infty}$ category with $m_1$ given by the differential, $m_2$ given by composition and trivial $m_i$'s for $i > 2$.

For any set $S$ there exists an $A_\infty$ category $\mathbb{C}\{S\}$. The objects of $\mathbb{C}\{S\}$ are elements of $S$ and
$$
\mathbb{C}\{S\}(s_1,s_2) = \left\{\begin{array}{ll}\mathbb{C} & \textrm{if } s_1 = s_2,\\
0 & \textrm{otherwise.}  \end{array} \right.
$$
All operations $m_n$ in $\mathbb{C}\{S\}$ are trivial.
\begin{dfn}
A functor of $A_{\infty}$ categories $F: \mathcal{A} \to \mathcal{B}$ is a map \mbox{$F_0: \textrm{ob}(\mathcal{A})\to \textrm{ob}(\mathcal{B})$} and a family of graded maps 
$$
F_n:\mathcal{A}(A_{n-1},A_{n})\otimes\ldots\otimes\mathcal{A}(A_0,A_1) \to \mathcal{B}(F_0(A_0), F_0(A_n))
$$
of degree $1-n$ such that 
$$
\sum_{r+s+t=n}(-1)^{r+st}F_{r+1+t}(\textrm{id}^{\otimes r}\otimes m_s \otimes \textrm{id}^{\otimes t}) = \sum_{i_1 + \ldots + i_r=n}(-1)^p m_r(F_{i_1}\otimes \ldots \otimes F_{i_r}), 
$$
where $p=(r - 1)(i_1 - 1) + (r - 2)(i_2 - 1) + \ldots + 2(i_{r-2} - 1) + (i_{r-1} - 1).$

Composition of functors is given by
$$
(F\circ G)_n = \sum_{i_1+\ldots+i_s = n}F_s\circ (G_{i_1}\otimes\ldots \otimes G_{i_s}).
$$
\end{dfn}

A functor $F:\mathcal{A}\to \mathcal{B}$ is called an \emph{$A_\infty$-quasi-isomorphism} if $F_1$ is a quasi-isomorphism. 

An $A_\infty$ category $\mathcal{A}$ is \emph{strictly unital} if for any object $A\in \textrm{ob}(\mathcal{A})$ there exists a morphisms $1_A \in \mathcal{A}(A,A)$ of degree 0 such that for any object $A'$ in $\mathcal{A}$ and any morphisms $\phi\in \mathcal{A}(A,A')$, $\psi \in \mathcal{A}(A',A)$ we have $m_2(\phi,1_A) =\phi$ and \mbox{$m_2(1_A,\psi) = \psi$}. Moreover, for $n\neq 2$ the operation $m_n$ equals 0 if any of its argument is equal to $1_A$. 

In particular, for any set $S$ the category $\mathbb{C}\{S\}$ is strictly unital.

The category $\mathcal{A}$ is \emph{homologically unital} if there exist units for the homotopy category $H(\mathcal{A})$.    

Lef\'evre-Hasegawa in \cite{bib_L} gives a connection between strictly and homologicaly unital categories.
\begin{prop}\label{thm:strict}
Minimal homolgically unital $A_\infty$ category is $A_\infty$-quasi-isomorphic to a minimal strictly unital $A_\infty$ category. 
\end{prop}

\begin{rem}\label{rem_strict}
A minimal $A_\infty$ category is equal to its homotopy category. Hence, an $A_\infty$-quasi-isomorphism $F$ given by the above proposition satisfies  $F_1 = \textrm{id}$.
\end{rem}

An $A_\infty$ category $\mathcal{A}$ is \emph{augmented} if there exists a strict unit preserving functor $\epsilon:\mathbb{C}\{\textrm{ob}(\mathcal{A})\} \to \mathcal{A}$. Then $\mathcal{A}$ decomposes as $\mathcal{A} = \mathbb{C}\{\textrm{ob}(\mathcal{A})\} \oplus \bar{\mathcal{A}}$.

\subsubsection*{Minimal model}\label{ssec:minimal}

Any $A_{\infty}$ category $\mathcal{A}$ is $A_\infty$-quasi-isomorphic to its homotopy category $H(\mathcal{A})$.

\begin{thm}\label{thm_minimal model}\emph{[Kadeishvili \cite{bib_K}]}
If $\mathcal{A}$ is an $A_{\infty}$ category, then $H(\mathcal{A})$ admits an $A_{\infty}$ category structure such that
\begin{enumerate}
\item $m_1=0$ and $m_2$ is induced from $m_2^A$ and
\item there is an $A_{\infty}$-quasi-isomorphism $\mathcal{A} \to H(\mathcal{A})$ inducing the identity on cohomology.
\end{enumerate}
Moreover, this structure is unique up to a non unique $A_{\infty}$-isomorphism.
\end{thm}

The $A_{\infty}$ category $H(\mathcal{A})$ is called the \textit{minimal model} of $\mathcal{A}$.

\begin{rem}\label{rem:unit}
Let $\mathcal{C}$ be a DG category. Its minimal model $H(\mathcal{C})$ is $A_\infty$-quasi-isomorphic to a strictly unital $A_\infty$ category. Indeed, the category $\mathcal{C}$ is strictly unital and hence homologically unital. Its homotopy category is also homologically unital and Proposition \ref{thm:strict} guarantees that there exists a stricly unital minimal category $A_\infty$-quasi-isomorphic to $H(\mathcal{C})$.  
\end{rem}
\subsubsection*{The universal DG category}

For any $A_\infty$ category $\mathcal{A}$ there exists a DG category $U(\mathcal{A})$ and an $A_\infty$-quasi-isomorphism $\mathcal{A} \to U(\mathcal{A})$. To define the category $U(\mathcal{A})$ we need the following definitions (see further \cite{bib_L}).
\begin{dfn}
A \emph{DG cocategory} $\mathcal{B}$ consists of 
\begin{itemize}
\item the set of objects $B_i\in \textrm{ob} (\mathcal{B})$, 
\item for any pair of objects $B_i,B_j\in \textrm{ob} (\mathcal{B})$ a complex of $\mathbb{C}$-vector spaces $\mathcal{B}(B_i,B_j)$ with a differential $d^{ij}$ of degree one and
\item a coassociative cocomposition -- a family of linear maps
$$
\Delta:\mathcal{B}(B_i,B_j) \to \sum_{B_k \in \textrm{ob}(\mathcal{B})} \mathcal{B}(B_k,B_j) \otimes \mathcal{B}(B_i, B_j).
$$
\end{itemize}
These data have to satisfy the condition
$$
\Delta \circ d = (d\otimes \textrm{id} + \textrm{id}\otimes d)\circ \Delta.
$$ 
\end{dfn}
For any set $S$ the $A_\infty$ category $\mathbb{C}\{S\}$ is also a DG cocategory.

A functor $\Phi$ between DG cocategories $\mathcal{B}$ and $\mathcal{B}'$ preserves the grading and differentials on morphisms and satisfies the condition
$$
\Delta\circ \Phi = (\Phi\otimes \Phi) \circ \Delta.
$$

A DG cocategory $\mathcal{B}$ is \emph{counital} if it admits a counit -- a functor \mbox{$\eta: \mathcal{B} \to \mathbb{C}\{\textrm{ob}(\mathcal{B})\} $.} The category $\mathcal{B}$ is \emph{coaugmented} if it is counital and admits a coaugmentation functor $\varepsilon: \mathbb{C}\{\textrm{ob}(\mathcal{B})\} \to \mathcal{B}$ such that the composition $\eta\varepsilon$ is the identity on $\mathbb{C}\{\textrm{ob}(\mathcal{B})\}$.

Let $\mathcal{B}$ be a coaugmented DG cocategory.Denote by $\bar{\mathcal{B}}$ a cocategory with the same objects as $\mathcal{B}$ and morphisms $\bar{\mathcal{B}}(B_i,B_j) = \textrm{ker }\varepsilon$.

For an augemnted $A_\infty$ category $\mathcal{A}$ one can define its \emph{bar} DG cocategory $B_\infty(\mathcal{A})$. Recall that as an augmented category  $\mathcal{A}$ can be written as \mbox{$ \bar{\mathcal{A}}\oplus \mathbb{C}\{\textrm{ob}(\mathcal{A})\}$.} Then $B_\infty(\mathcal{A}) = T^c(S\bar{\mathcal{A}})$ is a tensor cocategory of the suspension of $\bar{\mathcal{A}}$. Here $S\bar{\mathcal{A}}$ denotes the category $\bar{\mathcal{A}}$ with a shift in a morphisms spaces $(S\bar{\mathcal{A}})^n(A,A') = \bar{\mathcal{A}}^{n+1}(A,A')$. $S\bar{\mathcal{A}}$ is not an $A_\infty$ category, however the operations $m_n$ in $\bar{\mathcal{A}}$ define graded maps of degree 1 in $S\bar{\mathcal{A}}$.
\begin{align*}
b_n:S\bar{\mathcal{A}}(A_{n-1},A_{n})\otimes&\ldots\otimes S\bar{\mathcal{A}}(A_0,A_1) \to S\bar{\mathcal{A}}(A_0, A_n),\\
b_n &= -s\circ m_n\circ \omega^{\otimes n},
\end{align*}
where $s:V\to SV$ is a suspension of a graded vector space $V$ and $\omega = s^{-1}$.

The cocategory $ B_\infty(\mathcal{A})$ has the same objects as $\mathcal{A}$ and 
\begin{align*}
B_\infty(\mathcal{A})(A,A') & = S\bar{\mathcal{A}}(A,A') \oplus \bigoplus_{A_1\in \textrm{ob}(\mathcal{A})} S\bar{\mathcal{A}}(A_1,A')\otimes S\bar{\mathcal{A}}(A,A_1) \\
\oplus &\bigoplus_{A_1,A_2 \in \textrm{ob}(\mathcal{A})} S\bar{\mathcal{A}}(A_2,A') \otimes S\bar{\mathcal{A}}(A_1,A_2) \otimes S\bar{\mathcal{A}}(A,A_1) \oplus \ldots
\end{align*}
for $A\neq A'$. In the case $A=A'$ we have
\begin{align*}
B_\infty(\mathcal{A})(A,A) &= 1_A \oplus S\bar{\mathcal{A}}(A,A) \oplus \bigoplus_{A_1\in \textrm{ob}(\mathcal{A})} S\bar{\mathcal{A}}(A_1,A)\otimes S\bar{\mathcal{A}}(A,A_1)\\
& \oplus \bigoplus_{A_1,A_2 \in \textrm{ob}(\mathcal{A})} S\bar{\mathcal{A}}(A_2,A) \otimes S\bar{\mathcal{A}}(A_1,A_2) \otimes S\bar{\mathcal{A}}(A,A_1) \oplus \ldots
\end{align*}
where the degree of $1_A$ is zero.
To simplify the notation we shall write $(\alpha_n, \ldots, \alpha_1)$ for $\alpha_n\otimes \ldots\otimes \alpha_1$.
The differential in $B_\infty(\mathcal{A})$ is given by
\begin{align*}
&d(\alpha_n, \ldots,\alpha_1) =\\
& \sum_{k=1}^n \sum_{l=1}^{n-k+1} (-1)^{|\alpha_{l-1}| + \ldots + |\alpha_1|} ( \alpha_n, \ldots  \alpha_{l+k}, b_k(\alpha_{l+k-1},\ldots, \alpha_l),\alpha_{l-1},\ldots  \alpha_1),
\end{align*}
for $(\alpha_n, \ldots,\alpha_1) \in B_\infty(\mathcal{A})(A,A')$. The cocomposition is given by
\begin{align*}
&\Delta(\alpha_n, \ldots,\alpha_1)=\\
&1_{A'}\otimes (\alpha_n, \ldots,\alpha_1) + (\alpha_n, \ldots,\alpha_1)\otimes 1_A + \sum_{l=1}^{n-1}(\alpha_n,\ldots,\alpha_{l+1})\otimes (\alpha_l,\ldots, \alpha_1).
\end{align*}
With these definitions $B_\infty(\mathcal{A})$ is an augmented DG cocategory.

\begin{rem}\label{rem_barfinite}
Let $\mathcal{A}$ be an ordered and finite $A_\infty$ category such that $\bar{\mathcal{A}}(A,A) = 0$ for any object $A\in \textrm{ob}\mathcal{A}$. Then the DG cocategory $B_\infty(\mathcal{A})$ also satisfies these conditions; i.e. is ordered, finite and $\overline{B_\infty(\mathcal{A})}(A,A) = 0$ for any object $A$.
\end{rem}

Analogously, to an augmented DG cocategory $\mathcal{B}$ via a \emph{cobar} construction one can assign a DG category $\Omega(\mathcal{B})$. Let $\mathcal{B}$ be a DG cocategory with a differential $d$ and cocomposition $\Delta$. Its cobar DG category is equal to 
$T(S^{-1}\bar{\mathcal{B}})$. Here $S^{-1}\bar{\mathcal{B}}$ denotes the shift of the cocategory $\mathcal{B}$ and $T(S^{-1}\bar{\mathcal{B}})$ is the tensor DG category of it. As before the morphisms spaces in $T(S^{-1}\bar{\mathcal{B}})$ are given by
\begin{align*}
\Omega(\mathcal{B})(B,B') & = S^{-1}\bar{\mathcal{B}}(B,B') \oplus \bigoplus_{B_1\in \textrm{ob}(\mathcal{B})} S^{-1}\bar{\mathcal{B}}(B_1,B')\otimes S^{-1}\bar{\mathcal{B}}(B,B_1) \\
\oplus &\bigoplus_{B_1,B_2 \in \textrm{ob}(\mathcal{B})} S^{-1}\bar{\mathcal{B}}(B_2,B') \otimes S^{-1}\bar{\mathcal{B}}(B_1,B_2) \otimes S^{-1}\bar{\mathcal{B}}(B,B_1) \oplus \ldots
\end{align*}
for $B\neq B'$ and by
\begin{align*}
\Omega(\mathcal{B})(B,B) &= 1_B \oplus S^{-1}\bar{\mathcal{B}}(B,B) \oplus \bigoplus_{B_1\in \textrm{ob}(\mathcal{B})} S^{-1}\bar{\mathcal{B}}(B_1,B)\otimes S^{-1}\bar{\mathcal{B}}(B,B_1)\\
& \oplus \bigoplus_{B_1,B_2 \in \textrm{ob}(\mathcal{B})} S^{-1}\bar{\mathcal{B}}(B_2,B) \otimes S^{-1}\bar{\mathcal{B}}(B_1,B_2) \otimes S^{-1}\bar{\mathcal{B}}(B,B_1) \oplus \ldots
\end{align*}
for $1_B$ -- a morphism in degree zero. The composition in $\Omega(\mathcal{B})$ is defined by concatenation and the differential $\partial$ on the morphisms spaces is
$$
\partial = \sum 1\otimes\ldots \otimes 1\otimes( d + \Delta) \otimes 1\otimes\ldots\otimes 1.
$$

\begin{rem}\label{rem_cobarfinite}
If an ordered and finite DG cocategory $\mathcal{B}$ satisfies the condition $\bar{\mathcal{B}}(B,B)  = 0$, then the same is true about $\Omega(\mathcal{B})$.
\end{rem}

For an augmented $A_\infty$ category $\mathcal{A}$ its universal DG category $U(A)$ is defined as $\Omega(B_\infty(\mathcal{A}))$. There is a natural map $\mathcal{A}\to U(\mathcal{A})$. Lef\'evre-Hasegawa proves in \cite{bib_L} that this map extends to a functor and is an $A_\infty$-quasi-isomorphism. Moreover, for an $A_\infty$-quasi-isomorphisms $\phi$ the functor $U(\phi)$ is a quasi-isomorphism of DG categories.

\subsubsection*{$A_\infty$ modules}
\begin{dfn}
An $A_\infty$ module over an $A_\infty$ category $\mathcal{A}$ is an $A_\infty$ functor $M:\mathcal{A} \to \textrm{DGVect}_\mathbb{C}$. A morphism of modules $G:M\to N$ is given by a family $\{G_A:M_0(A) \to N_0(A)\}_{A\in \textrm{ob}\mathcal{A}}$ of morphisms in $\textrm{DGVect}_\mathbb{C}$ such that the diagrams
\[
\xymatrix{M_0(A_0) \ar[rr]^{M_n(\alpha_{n-1}\otimes\ldots\otimes\alpha_0)} \ar[d]^{G_{A_0}} & & M_0(A_n)\ar[d]^{G_{A_n}}\\
N_0(A_0)\ar[rr]^{N_n(\alpha_{n-1}\otimes\ldots\otimes\alpha_0)} && N_0(A_n)}
\]
commute for any $n\in \mathbb{N}$, $A_i\in \textrm{ob}\mathcal{A}$ and $\alpha_i \in \mathcal{A}(A_i,A_{i+1})$.
\end{dfn}
The category of $A_\infty$ modules over an $A_\infty$ category $\mathcal{A}$ will be denoted as $\textrm{Mod}_\infty^\textrm{strict}\mathcal{A}$. This notation agrees with the notation in \cite{bib_L}.

The morphism $G:M \to N$ of modules is an \emph{$A_\infty$-quasi-isomorphism} if $G_A$ is a quasi-isomorphism of complexes for any $A\in \textrm{ob}\mathcal{A}$.

The \emph{derived category} $D_\infty(\mathcal{A})$ of an $A_\infty$ category $\mathcal{A}$ is defined as a localization of the category $\textrm{Mod}_\infty^\textrm{strict}\mathcal{A}$ with respect to the class of $A_\infty$-quasi-isomorphisms. For $A_\infty$-quasi-isomorphic $A_\infty$ categories the derived categories are equivalent.

\begin{rem}[see \cite{bib_L}]\label{rem_DGderived}
For a DG category $\mathcal{C}$ the derived categories $D(\mathcal{C})$ and $D_\infty(\mathcal{C})$ are equivalent. It follows that for any augmented $A_\infty$ category $\mathcal{A}$ the derived category $D_\infty(\mathcal{A})$ is equivalent to the derived category of the universal algebra $D(U(\mathcal{A}))$.
\end{rem}

\begin{thm}\label{thm_finite}
Let $\mathcal{C}$ be a DG category with finitely many objects. Assume that $H(\mathcal{C})$ is an ordered and finite graded category such that $H(\mathcal{C})(C,C) = \mathbb{C}$ for any $C\in \mathcal{C}$. Then there exists an ordered and finite DG category $\widetilde{\mathcal{C}}$ such that $D(\mathcal{C})$ is equivalent to $D(\widetilde{\mathcal{C}})$.
\end{thm}
\begin{proof}
Theorem \ref{thm_minimal model} guarantees existence of a minimal model of $\mathcal{C}$ -- an $A_\infty$ category $H(\mathcal{C})$. By remark \ref{rem:unit} we can assume that $H(\mathcal{C})$ is strictly unital. As $H(\mathcal{C})(C,C) = \mathbb{C}$ for any $C$ the category $H(C)$ is an augmented oriented $A_\infty$ category.  We have  
$$
\mathcal{D} = D(\mathcal{C}_\sigma) = D_\infty(H(\mathcal{C}_\sigma) = D(\mathcal{U}(H(\mathcal{C}_\sigma)).   
$$
Remarks \ref{rem_barfinite} and \ref{rem_cobarfinite} show that the category $\mathcal{U}(H(\mathcal{C}_\sigma))$ is ordered and finite.
\end{proof}

\begin{cor}\label{cor_finite}
Let $X$ be a smooth projective variety and let $\sigma = \langle \mathcal{E}_1, \ldots, \mathcal{E}_n \rangle$ be a full exceptional collection on $X$. There exists an oriented, finite DG category $\widetilde{\mathcal{C}_\sigma}$ such that $D^b(X)$ is equivalent to $D^b(\widetilde{\mathcal{C}_\sigma})$.
\end{cor}
\begin{proof}
By theorem \ref{thm:generating} there exists a DG category $\mathcal{C}_\sigma$ such that \mbox{$D^b(X) = D^b(\mathcal{C}_\sigma)$}. As sheaves $\mathcal{E}_1, \ldots, \mathcal{E}_n$ are exceptional and the category $\textrm{Coh}(X)$ is Ext-finite, the category $\mathcal{C}_\sigma$ satisfies conditions of theorem \ref{thm_finite} and there exists a DG category $\widetilde{\mathcal{C}_\sigma}$ such that categories $D(\mathcal{C}_\sigma)$ and $D(\widetilde{\mathcal{C}_\sigma)}$ are equivalent. As $D^b(\mathcal{C}_\sigma)\subset D(\mathcal{C}_\sigma)$ is the subcategory of compact objects it follows that $D^b (\mathcal{C}_\sigma) = D^b( \widetilde{ \mathcal{C}_\sigma)}$.
\end{proof}

The category $\widetilde{\mathcal{C}_\sigma}$ is $A_\infty$-quasi-isomorphic to the category $\mathcal{C}_\sigma$ and hence $H^k(\widetilde{\mathcal{C}_\sigma}(\mathcal{E}_i, \mathcal{E}_j)) = \textrm{Ext}^k_X(\mathcal{E}_i, \mathcal{E}_j)$.

To simplify the notation we will denote by $\mathcal{C}_\sigma$ an ordered finite DG category associated to a full exceptional collection $\sigma$.

\subsection*{Action of the braid group}

We have seen that the braid group acts on the set of exceptional collection. This action lifts to associated DG categories.

Twisted complexes provide a description of the categories $\mathcal{C}_{L_i\sigma}$ and $\mathcal{C}_{R_i\sigma}$ by means of $\mathcal{C}_\sigma$. To see it we need to define a tensor product of a twisted complex with a complex of vector spaces. 

Let $\mathcal{C}$ be a finite DG category, \mbox{$C\in \mathcal{C}^\textrm{pre-tr}$} be a twisted complex and let $V^\bullet$ be a finite dimensional complex of vector spaces with the differential \mbox{$\partial^i:V^i\to V^{i+1}$.
$C\otimes V$} is defined as $(\bigoplus_{\{i|V^i \neq 0\}} C[-i]^{\oplus \textrm{dim}V^i},q_{i,j}) \in (\mathcal{C}^\textrm{pre-tr})^\textrm{pre-tr}$. The morphisms $q_{i,i+1}$ are induced by the differential $\partial^i$ tensored with the identity on $C$ and $q_{i,j}=0$ for $j\neq i+1$.

Now, let $C,D\in \mathcal{C}^\textrm{pre-tr}$ be twisted complexes. There exist closed morphisms of degree 0 
\begin{align*}
\phi:& C\otimes \mathcal{H}om_{\mathcal{C}^\textrm{pre-tr}}(C,D) \to D,\\
\psi:& C \to \mathcal{H}om_{\mathcal{C}^\textrm{pre-tr}}(C,D)^* \otimes D.
\end{align*}
The morphism $\phi_{i,0}: C[-i]^{\oplus \textrm{dim Hom}^i(C,D)} \to D$ is given by morphisms of degree $i$ between $C$ and $D$; $\psi$ is defined analogously. 
Now we define new twisted complexes over $\mathcal{C}$
\begin{align*}
L_CD & = \textrm{Tot}(C(\phi)[-1]),\\
R_DC &= \textrm{Tot}(C(\psi)),
\end{align*}
where $C(\phi)$ denotes the cone of $\phi$.

Let $\sigma = \langle \mathcal{E}_1,\ldots,\mathcal{E}_n\rangle$ be a full exceptional collection on a smooth projective variety $X$ and let $\mathcal{C}_\sigma$  be the category described in the theorem \ref{thm_finite}. Let denote by $E_1,\ldots, E_n$ objects of $\mathcal{C}_\sigma$. We define two full subcategories of $\mathcal{C}_\sigma^\textrm{pre-tr}$; $\mathcal{C}_\sigma^{L_i}$ with objects $E_1, \ldots, E_i, L_{E_i}E_{i+1},E_{i+2},\ldots E_n$ and $\mathcal{C}_\sigma^{R_i}$ with $E_1,\ldots, E_{i-1}, R_{E_{i+1}}E_i, E_{i+1},\ldots, E_n$.

\begin{prop}
Let $\sigma = \langle \mathcal{E}_1, \ldots, \mathcal{E}_n \rangle$ be a full exceptional collection on $X$ and let $\mathcal{C}_\sigma$ be a finite DG category with objects $E_1, \ldots, E_n$ with $H^k\mathcal{C}_\sigma(E_i, E_j) = \textrm{Ext}^k_X(\mathcal{E}_i, \mathcal{E}_j)$ and such that $D^b(X)$ is equivalent to $D^b(\mathcal{C}_\sigma)$. Then the categories $\mathcal{C}_\sigma^{L_i}$ and $\mathcal{C}_\sigma^{R_i}$ satisfy analogous conditions for the collections $L_i\sigma$ and $R_i \sigma$ respectively. 
\end{prop} 

\begin{proof}
As the category $\mathcal{C}_\sigma$ is finite, mutations of twisted complexes over $\mathcal{C}_\sigma$ are well defined. Furthermore, the category $D^b(X)$ is equivalent to $\mathcal{C}_\sigma^\textrm{tr}$. Under this equivalence $L_{\mathcal{E}_i}\mathcal{E}_{i+1}$ corresponds to $L_{E_i}E_{i+1}$. Hence the category $\mathcal{C}_\sigma^{L_i}$ is the DG category described by theorem \ref{thm:generating}. By construction it is also finite.
\end{proof}

\section{Calculating the category $\mathcal{C}_\sigma$}
First, let us define universal extensions after \cite{bib_HP}.

\subsection*{Universal extensions}
Let $E$, $F$ be objects of a $\mathbb{C}$-linear abelian category. Note that $\textrm{End}(\textrm{Ext}^1(E,F)) = \textrm{Ext}^1(E, F \otimes \textrm{Ext}^1(E,F)^*)$. As $\textrm{id}\in \textrm{End}(\textrm{Ext}^1(E,F))$ there exists a distinguished element $\widetilde{\textrm{id}}$ of  \mbox{$\textrm{Ext}^1(E, F \otimes \textrm{Ext}^1(E,F)^*)$.}

$\bar{E}$ -- the \emph{universal extension of $E$ by $F$} is defined as the extension of $E$ by $F \otimes \textrm{Ext}^1(E,F)^*$ corresponding to $\widetilde{\textrm{id}} \in \textrm{Ext}^1(E, F \otimes \textrm{Ext}^1(E,F)^*)$. $\bar{E}$ fits into the short exact sequence:

\begin{equation}
\xymatrix{0 \ar[r] & F\otimes \textrm{Ext}^1(E,F)^* \ar[r] & \bar{E} \ar[r] & E \ar[r] &0.}
\end{equation}

This short exact sequence gives the long exact sequence 
\[
\xymatrix{0\ar[r] & \textrm{Hom}(E,F) \ar[r] & \textrm{Hom}(\bar{E},F) \ar[r] & \textrm{Hom}(F,F) \otimes \textrm{Ext}^1(E,F) \ar[r] & \dots}.
\]
It shows that Ext$^1(\bar{E},F) = 0$ and the groups Ext$^n(\bar{E},F)$ and Ext$^n(E,F)$ are isomorphic for $n>1$ if $\textrm{Ext}^n(F,F) = 0$ for $n>0$.

Moreover, if Ext$^n(F,E) = 0$ for $n\geq 0$, Hom$(F,F) = \mathbb{C}$ and Ext$^q(F,F)=0$ for $q>0$ then the long exact sequence
\[
\xymatrix{0\ar[r] & \textrm{Hom}(F,F) \otimes \textrm{Ext}^1(E,F)^* \ar[r] & \textrm{Hom}(F,\bar{E}) \ar[r] & \textrm{Hom}(F,E) \ar[r] &\ldots }
\] 
shows that $\textrm{Hom}(F,\bar{E}) = \textrm{Ext}^1(E,F)^*$.

Thus, if $F$ is exceptional and Ext$^n(F,E)=0$ for all $n$, the objects $\bar{E}$, $F$ and morphisms between them determine $E$ -- as a cone of the canonical morphism 
\[
\xymatrix{ F\otimes \mathcal{H}om(F,\bar{E}) \ar[r]^(.7){\textrm{can}} & \bar{E} \ar[r] & E.}
\]

Assume that $(E,F)$ is an exceptional pair in an enhanced triangulated category with the enhancement $\widetilde{\mathcal{C}}$ and denote by $\mathcal{C}$ the DG subcategory of $\widetilde{\mathcal{C}}$ with objects $\bar{E}$ and $F$. Then the cone of the canonical morphism \mbox{$F\otimes \mathcal{H}om(F,\bar{E}) \rightarrow \bar{F}$} in $\mathcal{C}^\textrm{pre-tr}$ corresponds to $E$. Hence, as in the case of mutations, knowing the DG subcategory of $\widetilde{\mathcal{C}}$ with objects $\bar{E}$ and $F$ we can calculate the DG subcategory with objects $E$ and $F$.    

\begin{thm}\label{thm:universal}
Let $\sigma = \langle \mathcal{E}_1, \ldots, \mathcal{E}_n \rangle$ be a full exceptional collection on a smooth projective variety $X$ such that $\textrm{Ext}^k(\mathcal{E}_i, \mathcal{E}_j) = 0$ for $k>1$ and any $i$, $j$. 
The DG category $\mathcal{C}_\sigma$ can be calculated by means of universal extensions.
\end{thm}

\begin{proof}
Following \cite{bib_HP} we define $\mathcal{E}_i(1) = \mathcal{E}_i(2) =\ldots = \mathcal{E}_i(i) = \mathcal{E}_i$ and for $j>i$ we put $\mathcal{E}_i(j)$ to be the universal extension of $\mathcal{E}_i(j-1)$ by $\mathcal{E}_j$.
\[
\xymatrix{0 \ar[r] & \mathcal{E}_j\otimes \textrm{Ext}^1(\mathcal{E}_i(j-1), \mathcal{E}_j)^* \ar[r] & \mathcal{E}_i(j) \ar[r] & \mathcal{E}_i(j-1) \ar[r] &0.}
\]

Then $\mathcal{F} = \bigoplus_{i=1}^n \mathcal{E}_i(n)$ is a tilting object -- it generates $D^b(X)$ and Ext$^i(\mathcal{F}, \mathcal{F}) = 0$ for $i>0$. Moreover, as the $\mathcal{E}_i$'s are exceptional and there are no morphisms from $\mathcal{E}_j$ to $\mathcal{E}_i(j-1)$, the objects $\mathcal{E}_1(n),\ldots,\mathcal{E}_n(n)$ determine $\mathcal{E}_i$'s. Thus, as described above, the endomorphisms algebra of $\bigoplus_{i=1}^n \mathcal{E}_i(n)$ determines the DG structure of the collection $(\mathcal{E}_1,\ldots, \mathcal{E}_n)$.
\end{proof}

\section{Exceptional collections of line bundles on rational surfaces}

After \cite{bib_HP} we describe full exceptional collections of line bundles on rational surfaces.
\subsection*{Construction and mutations}
Let $X$ be a rational surface. Then $X$ is a blow up of $\mathbb{P}^2$ or of the Hirzebruch surface $\mathbb{F}_a$. There is a sequence of blow-ups
\[
\xymatrix{X=X_t \ar[r]^{\pi_t} & X_{t-1} \ar[r]^{\pi_{t-1}} & \ldots \ar[r]^{\pi_1} & X_0,}
\] 
where $\pi_i$ has an exceptional divisor $E_i$ and $X_0$ is either the projective plane $\mathbb{P}^2$ or the Hirzebruch surface $\mathbb{F}_a$.

Let $R_i$ denote preimages of the divisor $E_i$ to $X_k$ for $k\geq i$. Then $R_1,\ldots,R_t$ are mutually orthogonal of self intersection -1.

Full exceptional collections of line bundles on $X$ can be obtained by an augmentation from an exceptional collection on $X_0$. To describe the augmentation we identify a line bundle $\mathcal{L}$ on $X_i$ with its pull back via $\pi_j$'s and denote them by the same letter.

Let $\sigma=\langle\mathcal{L}_1,\ldots,\mathcal{L}_s\rangle$ be an exceptional collection on $X_i$. The augmentation of $\sigma$ is \mbox{$\sigma '=\langle \mathcal{L}_1(R_{i+1})\ldots,\mathcal{L}_{k-1}(R_{i+1}),\mathcal{L}_k, \mathcal{L}_k(R_{i+1}), \mathcal{L}_{k+1}, \ldots \mathcal{L}_s \rangle$} -- an exceptional collection on $X_{i+1}$. 

Hille and Perling in \cite{bib_HP} proved that all the above described collections are full and have no nontrivial Ext$^2$ groups between elements. Moreover, mutations allow to present every one of them in the following form.
 
\begin{prop}
Any exceptional collection of line bundles on $X$ obtained via augmentation can be mutated to $\langle \mathcal{O}, \mathcal{O}(R_t), \ldots, \mathcal{O}(R_1),\mathcal{L}_1, \ldots, \mathcal{L}_r \rangle $, where $\langle \mathcal{O}, \mathcal{L}_1, \ldots, \mathcal{L}_r \rangle$  is an exceptional collection on $X_0$.
\end{prop}

\begin{proof}
As $\mathcal{O}_X =(\pi_i\ldots \pi_t)^* (\mathcal{O}_{X_i})$ and the line bundle $\mathcal{O}_X(R_i)$ is a pullback of $\mathcal{O}_{X_i}(E_i)$, the projection formula gives
\begin{align*}
H^q(X,\mathcal{O}(R_i)) = H^q(X_i, \mathcal{O}(E_i)) & = \left\{\begin{array}{ll} \mathbb{C} & q=0,\\ 0 & q\neq 0, \end{array} \right.\\
H^q(X,\mathcal{O}(2R_i)) = H^q(X_i, \mathcal{O}(2E_i)) & = \left\{\begin{array}{ll} \mathbb{C} & q=0,1, \\ 0 & q\neq 0,1, \end{array} \right.\\
H^q(X,\mathcal{O}(-R_i)) = H^q(X_i, \mathcal{O}(-E_i))& = 0.
\end{align*}
Short exact sequences 
\[
\xymatrix{ 0 \ar[r]& \mathcal{O}(-R_i)\ar[r] & \mathcal{O} \ar[r] &  \mathcal{O}_{R_i}\ar[r] & 0,\\
0 \ar[r] & \mathcal{O}\ar[r] & \mathcal{O}(R_i) \ar[r] & \mathcal{O}_{R_i}(R_i) \ar[r] &  0}
\]
show that 
\begin{align*}
\textrm{Ext}^q(\mathcal{O}, \mathcal{O}_{R_i}) & = \left\{\begin{array}{ll}\mathbb{C} & q=0, \\ 0 & q \neq 0, \end{array} \right.\\
\textrm{Ext}^q(\mathcal{O}_{R_i}(R_i),\mathcal{O}) & =  \left\{\begin{array}{ll} \mathbb{C} & q=1, \\ 0 & q\neq 1. \end{array}  \right.
\end{align*}

The collection on $X$ obtained via augmentation is of the form $\langle \mathcal{L}_1(R_t),\ldots,\mathcal{L}_{i-1}(R_t),\mathcal{L}_i, \mathcal{L}_i(R_t), \mathcal{L}_{i+1}, \ldots \mathcal{L}_s \rangle$, where $\mathcal{L}_j$'s are pull backs of line bundles on $X_{t-1}$.  

Equality
$$
 \textrm{Hom}(\mathcal{L}_i, \mathcal{L}_i(R_t)) = \textrm{Hom}(\mathcal{O},\mathcal{O}(R_t)) = \mathbb{C}
$$
and the short exact sequence
$$
0\rightarrow \mathcal{L}_i \rightarrow \mathcal{L}_i(R_t) \rightarrow \mathcal{O}_{R_t}(R_t) \rightarrow 0
$$ 
show that this collection can be mutated to $\langle \mathcal{L}_1(R_t), \ldots, \mathcal{L}_{i-1}(R_t), \mathcal{O}_{R_t}(R_t), \mathcal{L}_i,$ $ \mathcal{L}_{i+1}, \ldots \mathcal{L}_s \rangle$.

Then 
$$
\textrm{Hom}(\mathcal{L}_k(R_t), \mathcal{O}_{R_t}(R_t)) = \textrm{Hom}(\mathcal{O}(R_t), \mathcal{O}_{R_t}(R_t)) = \textrm{Hom}(\mathcal{O}, \mathcal{O}_{R_t}) = \mathbb{C}
$$
and the exact sequences
$$
0\rightarrow \mathcal{L}_k \rightarrow \mathcal{L}_k(R_t) \rightarrow \mathcal{O}_{R_t}(R_t) \rightarrow 0
$$
provide further mutations to $\langle \mathcal{O}_{R_t}(R_t), \mathcal{L}_1, \ldots, \mathcal{L}_s \rangle$. 

The collection $\langle \mathcal{L}_1,\ldots, \mathcal{L}_s \rangle$ is a pull back of a collection on $X_{t-1}$ and it again has the form $\langle \mathcal{L}'_1 (R_{t-1}), \ldots, \mathcal{L}'_{k-1}(R_{t-1}), \mathcal{L}'_k, \mathcal{L}'_k(R_{t-1}), \mathcal{L}'_{k+1}, \ldots, \mathcal{L}'_{s-1} \rangle$ for some $k$. As before, it can be mutated to $\langle \mathcal{L}'_1(R_{t-1}), \ldots, \mathcal{L}'_{k-1}(R_{t-1}),$ $ \mathcal{O}_{R_{t-1}}(R_{t-1}),\mathcal{L}'_k, \ldots, \mathcal{L}'_{s-1}\rangle$ and then to $\langle \mathcal{O}_{R_{t-1}}(R_{t-1}),\mathcal{L}'_1, \ldots, \mathcal{L}'_{s-1} \rangle$.

Continuing, we can mutate the collection on $X$ to $\langle \mathcal{O}_{R_t}(R_t), \ldots, \mathcal{O}_{R_1}(R_1),$ $ \mathcal{O}, \mathcal{L}_1, \ldots, \mathcal{L}_r \rangle$. And this collection can be mutated to $\langle \mathcal{O}, \mathcal{O}(R_t), \ldots, \mathcal{O}(R_1),$ $\mathcal{L}_1, \ldots, \mathcal{L}_r \rangle $.
\end{proof}

\subsection*{Example}

Let $X$ be the blow up of $\mathbb{P}^2$ along the degree 2 subscheme supported at a point $x_0$. The Picard group of $X$ is generated by divisors $E_1$, $E_2$ and $H$ with intersection form given by
\begin{align*}
&E_1^2 = -2, \quad E_2^2 = -1, \quad H^2 =1, \\
&E_1E_2 = 1, \quad E_iH = 0.\\
\end{align*}
In the notation of the previous section we have
\begin{align*}
&R_1 = E_1 +E_2, \quad R_2 = E_2.
\end{align*}
The collection $\langle \mathcal{O}_X, \mathcal{O}_X(E_2), \mathcal{O}_X(E_1+E_2), \mathcal{O}_X(H), \mathcal{O}_X(2H) \rangle$ is full and exceptional with a quiver 
\[
\xymatrix{\mathcal{O}_X \ar[r]^{\alpha} \ar@/_2pc/[rrr]_{\eta} & \mathcal{O}_X(E_2) \ar[r]^{\beta} \ar@/^1pc/[r]^{\bar{\beta}}& \mathcal{O}_X(E_1+E_2) \ar[r]^{\gamma_1} \ar@<-1ex>[r]_{\gamma_2} & \mathcal{O}_X(H) \ar@<2ex>[r]^{\delta_1} \ar[r]^{\delta_2} \ar@<-2ex>[r]^{\delta_3}& \mathcal{O}_X(2H)}
\]
with $\bar{\beta}$ in degree 1 and relations
\begin{align*}
&\delta_1\gamma_2 = \delta_2\gamma_1,&  &\delta_1 \eta = \delta_3\gamma_1\beta\alpha,&  \\
&\delta_2\eta = \delta_3\gamma_2\beta\alpha,&  &\bar{\beta}\alpha = 0,&  \\
&\gamma_1\bar{\beta}  = 0,& &\gamma_2\bar{\beta}  = 0.&
\end{align*}

Let $D_2$ denote a strict transform of a line on $\mathbb{P}^2$ such that $D_2 \cap E_2\neq \emptyset$. The morphism $\gamma_2$ in the above quiver is zero along $D_2 + E_2$. $\alpha$ has zeros along $E_2$ and $\beta$ -- along $E_1$. These three morphisms are determined uniquely up to a constant. On the other hand, $\gamma_1$ is zero along $D_1$ -- a strict transform of a line on $\mathbb{P}^2$ such that $D_1 \cap E_1 \neq \emptyset$. The divisor $D_1$ is not determined uniquely and one can change $\gamma_1$ in the above quiver by adding some multiplicity of $\gamma_2$. 

Note that morphisms form $\mathcal{O}_X$ to $\mathcal{O}_X(H)$ are pull backs of sections of $\mathcal{O}_{\mathbb{P}^2}(H)$. As the pull back of $x_0$ in $X$ is $E_1 \cup E_2$, a section of $\mathcal{O}_X(H)$ on $X$ is either zero on both $E_1$ and $E_2$ (possibly with some multiplicities) or it is nonzero on every point of $E_1$ and $E_2$. Sections having zeros along $E_1$ and $E_2$ are linear combinations of $\gamma_1\beta \alpha$ and $\gamma_2 \beta \alpha$. 

The morphisms $\eta$ in the quiver is any section of $\mathcal{O}_X(H)$ which is nonzero on $E_1$ and $E_2$. Hence, it can be changed by adding any linear combination of $\gamma_1\beta\alpha$ and $\gamma_2 \beta\alpha$.

To obtain the relations given, we note that $\textrm{Hom}(\mathcal{O}_X,\mathcal{O}_X(H)) = \textrm{Hom}(\mathcal{O}_X(H), \mathcal{O}_X(2H))$. Then we put $\delta_1 \leftrightarrow \gamma_1\beta\alpha$, $\delta_2 \leftrightarrow \gamma_2\beta\alpha$ and $\delta_3 \leftrightarrow \eta$.

To sum up, $\alpha$, $\beta$ and $\gamma_2$ are determined uniquely. Other morphisms can change by
\begin{align*}
\gamma_1 & \rightsquigarrow \gamma_1 + a\gamma_2\\
\eta &\rightsquigarrow \eta + b\gamma_1\beta\alpha + c\gamma_2\beta \alpha
\end{align*}
for $a,b,c\in \mathbb{C}$. These morphism determine $\delta_i$'s in such a way that the above relations are satisfied. 

In order to calculate the DG quiver of this collection, we have to calculate the endomorphisms algebra of a collection $\langle \mathcal{O}_X, V, \mathcal{O}_X(E_1+E_2), \mathcal{O}_X(H), \mathcal{O}_X(2H) \rangle$, where $V$ is defined by a non-trivial extension of $\mathcal{O}_X(E_2)$ by $\mathcal{O}_X(E_1+E_2)$.
\[
\xymatrix{0 \ar[r]& \mathcal{O}_X(E_1+E_2) \ar[r]^(0.6){\phi_1}& V \ar[r]^(0.4){\phi_2}& \mathcal{O}_X(E_2) \ar[r] & 0.}
\]

The collection of sheaves $\langle \mathcal{O}_X, V, \mathcal{O}_X(E_1+E_2), \mathcal{O}_X(H), \mathcal{O}_X(2H) \rangle$ is no longer exceptional. However, we can still draw the quiver of morphisms between objects. Then we can present $\mathcal{O}_X(E_2)$ as a complex
$$
\{\mathcal{O}_X(E_1+E_2) \otimes \textrm{Hom}(\mathcal{O}_X(E_1+E_2), V) \to V\}
$$
and calculate the DG quiver of the collection $\langle \mathcal{O}_X, \mathcal{O}_X(E_2), \mathcal{O}_X(E_1+E_2),$ $ \mathcal{O}_X(H), \mathcal{O}_X(2H) \rangle$  

The quiver of $(\langle \mathcal{O}_X, V, \mathcal{O}_X(E_1+E_2), \mathcal{O}_X(H), \mathcal{O}_X(2H) \rangle$ is 
\[
\xymatrix{\mathcal{O}_X \ar[r]^{\zeta} \ar@/^2pc/[rrr]^{\eta} & V \ar[r]^(0.4){\beta\phi_2} \ar@/_2pc/[rr]^{\iota_1} \ar@/_2pc/@<-2ex>[rr]_{\iota_2}& \mathcal{O}_X(E_1+E_2) \ar@<1ex>[l]^(0.4){\phi_1}  & \mathcal{O}_X(H) \ar@<2ex>[r]^{\delta_1} \ar[r]^{\delta_2} \ar@<-2ex>[r]^{\delta_3} & \mathcal{O}_X(2H)}
\]

Here $\zeta: \mathcal{O}_X \to V$ is such a morphism that $\phi_2 \zeta = \alpha$ and $\iota_i:V \to \mathcal{O}_X(H)$ are such that $\iota_i \phi_1 = \gamma_i$. Then $\iota_1 \phi_1 \beta\phi_2 \zeta = \gamma_1\beta\alpha$ and $\iota_2 \phi_1 \beta \phi_2 \zeta = \gamma_2 \beta \alpha$. 

$\zeta$, $\iota_1$ and $\iota_2$ are not determined uniquely, they can change by
\begin{align*}
\zeta & \rightsquigarrow \zeta + d\phi_1\beta \alpha,\\
\iota_1 & \rightsquigarrow \iota_1 + e\gamma_1\beta \phi_2 + f\gamma_2\beta \phi_2,\\
\iota_2 & \rightsquigarrow \iota_2 + g\gamma_1\beta \phi_2 + h\gamma_2\beta \phi_2
\end{align*}
for $d,\ldots, h\in \mathbb{C}$.

The obvious relations in this quiver are:
\begin{align*}
& \delta_2 \eta = \delta_3 \iota_2 \phi_1 \beta \phi_2 \zeta & \delta_1 \eta = \delta_3 \iota_1 \phi_1 \beta \phi_2 \zeta.
\end{align*}

Moreover, the exact sequence
\begin{align*}
0 \rightarrow \textrm{Hom}(\mathcal{O}_X(E_2), \mathcal{O}_X(2H)) &\rightarrow \textrm{Hom}(V,\mathcal{O}_X(2H))\\
&\rightarrow \textrm{Hom}(\mathcal{O}_X(E_1+E_2), \mathcal{O}_X(2H)) \rightarrow 0
\end{align*}
shows that
$$
\delta_1 \iota_2 - \iota_2 \delta_1 \in \textrm{span}\{\delta_1 \gamma_1 \beta \phi_2, \delta_1 \gamma_2 \beta \phi_2, \delta_2 \gamma_2 \beta \phi_2, \delta_3 \gamma_1 \beta \phi_2, \delta_3 \gamma_2 \beta \phi_2\}.
$$

One has to calculate are the compositions $\iota_1\zeta$ and $\iota_2\zeta$. 

\subsubsection*{Calculation of Massey products}

The short exact sequence
\[
\xymatrix{0 \ar[r] & \mathcal{O}_X(E_1+E_2) \ar[r]^(0.7){\phi_1} & V \ar[r]^(0.4){\phi_2} & \mathcal{O}_X(E_2) \ar[r] & 0}
\] 
gives $c_1(V) = E_1 + 2E_2$ and $c_2(V) = 0$. Let us consider the short exact sequence
\[
\xymatrix{0 \ar[r] & \mathcal{O}_X \ar[r]^\zeta & V \ar[r]^{\psi} & \mathcal{F} \ar[r] & 0. }
\]
Let $T\hookrightarrow \mathcal{F}$ be the torsion part of $\mathcal{F}$. Then we have the commutative diagram
\[
\xymatrix{ & & 0&0 & \\
&0 \ar[r] &\mathcal{F}/T \ar[r]^= \ar[u] & \mathcal{F}/T \ar[u]&\\
0 \ar[r] &\mathcal{O}_X \ar[u] \ar[r]^\zeta & V \ar[r]^\psi \ar[u]&\mathcal{F}\ar[r] \ar[u] & 0\\
0 \ar[r]& \mathcal{O}_X \ar[r] \ar[u]^=& \mathcal{G} \ar[r] \ar[u]^\chi &T\ar[r] \ar[u] &0\\
& &0 \ar[u] &0 \ar[u] &\\}
\] 
in which $\mathcal{G}$ is a sheaf of rank 1. $\mathcal{G}$ fits into a short exact sequence
$$
0 \rightarrow \mathcal{G} \rightarrow V \rightarrow \mathcal{F}/T \rightarrow 0
$$
with $V$ locally free and $\mathcal{F}/T$ torsion-free. Hence, $\mathcal{G}$ is torsion-free and it injects into its double dual $\mathcal{G}^{**}$ with cokernel $T'$,
$$
0 \rightarrow \mathcal{G} \rightarrow \mathcal{G}^{**} \rightarrow T' \rightarrow 0.
$$
$T'$ is a torsion sheaf and a subsheaf of $\mathcal{F}/T$, hence $T' = 0$ and $\mathcal{G}$ is reflexive.  Hartshorne in \cite{bib_H} proves that every reflexive sheaf on a smooth surface is locally free and therefore we obtain that $\mathcal{G} = \mathcal{O}_X(D)$ for some effective divisor $D$. The composition of morphisms 
\[
\xymatrix{\mathcal{O}_X \ar[r] & \mathcal{O}_X(D) \ar[r]^\chi & V \ar[r]^(.3){\psi_2} & \mathcal{O}_X(E_2)}
\]
is equal to $\beta$, so the morphism $\mathcal{O}_X(D) \rightarrow \mathcal{O}_X(E_2)$ is nonzero. It follows that $D$ is equal either to $0$ or to $E_2$. If $D = E_2$ then we would have a splitting
\[
\xymatrix{0 \ar[r] & \mathcal{O}_X(E_1 +E_2) \ar[r]^(.7){\phi_1} & V \ar[r]^{\phi_2}& \mathcal{O}_X(E_2) \ar[r] \ar@/_2pc/[l]_{\chi}& 0.}
\]
As $V$ is a nontrivial extension of $\mathcal{O}_X(E_2)$ by $\mathcal{O}_X(E_1+E_2)$ we get a contradiction. Hence, $D = 0$ and $\mathcal{F}$ is a torsion-free sheaf of rank 1. As $c_2(\mathcal{F}) =0$, $\mathcal{F}$ is a line bundle. $c_1(\mathcal{F}) = E_1+2E_2$ shows that $\mathcal{F} = \mathcal{O}_X(E_1+2E_2)$. Hence, $\zeta$ fits into the short exact sequence
\[
\xymatrix{0 \ar[r] & \mathcal{O}_X \ar[r]^\zeta & V \ar[r]^(.3)\psi & \mathcal{O}_X(E_1+2E_2) \ar[r] & 0.}
\]
Let $\mathcal{L}_i = \textrm{ker}(\iota_i)$, $\mathcal{K}_i =\textrm{coker}(\iota_i)$ and $\mathcal{M}_i =\textrm{Im}(\iota_i)$. There are three short exact sequences:
\[
\xymatrix{
0 \ar[r] &  \mathcal{L}_i \ar[r]& V \ar[r] & \mathcal{M}_i \ar[r] & 0,\\
0\ar[r]& \mathcal{M}_i \ar[r]& \mathcal{O}_X(H)\ar[r]& \mathcal{K}_i \ar[r]& 0,\\
0 \ar[r] & \mathcal{O}_X(E_1+E_2) \ar[r] & \mathcal{M}_i \ar[r] & \mathcal{N}_i \ar[r]& 0.}
\]
The diagram
\[
\xymatrix{ & 0 & 0 & 0 & \\
0 \ar[r] & \mathcal{O}_X(E_1+ E_2) \ar[r] \ar [u]& \mathcal{M}_i\ar[u] \ar[r] & \mathcal{N}_i \ar[r] \ar[u] & 0\\
0 \ar[r] & \mathcal{O}_X(E_1 + E_2) \ar[r] \ar[u]_=& V \ar[r] \ar[u] & \mathcal{O}_X(E_2) \ar[r] \ar[u]& 0 \\
& 0\ar[u] \ar[r] & \mathcal{L}_i \ar[r]^= \ar[u] & \mathcal{L}_i \ar[u]&\\
& & 0\ar[u] & 0 \ar[u] & }
\]

gives relations between the Chern classes:
\begin{align*}
c_1(V)& = E_1 + 2E_2= c_1(\mathcal{L}_i) + c_1(\mathcal{M}_i),\\
c_2(V) &= 0 = c_1(\mathcal{L}_i)c_1(\mathcal{M}_i) + c_2(\mathcal{M}_i),\\
H &= c_1(\mathcal{M}_i) + c_1(\mathcal{K}_i),\\
c_2(\mathcal{O}_X(H)) &=0 = c_1(\mathcal{M}_i)c_1(\mathcal{K}_i) + c_2(\mathcal{M}_i) + c_2(\mathcal{K}_i),\\
c_1(\mathcal{M}_i) &= E_1 + E_2 + c_1(\mathcal{N}_i),\\
c_2(\mathcal{M}_i) &= (E_1+E_2)c_1(\mathcal{N}_i) + c_2(\mathcal{N}_i),\\
E_2 &= c_1(\mathcal{L}_i) + c_1(\mathcal{N}_i),\\
0 &= c_1(\mathcal{L}_i)c_1(\mathcal{N}_i) + c_2(\mathcal{N}_i).
\end{align*}

Diagrams
\[
\xymatrix{& &0 &0 & \\
&0 \ar[r] & \mathcal{K}_1 \ar[r]^= \ar[u]& \mathcal{K}_1 \ar[u] & \\
0 \ar[r]&\mathcal{O}_X(E_1+E_2) \ar[r] \ar[u] &\mathcal{O}_X(H)\ar[r] \ar[u] &\mathcal{O}_{D_1}(H) \ar[r] \ar[u] & 0\\
0 \ar[r]& \mathcal{O}_X(E_1+E_2) \ar[r] \ar[u]^= & \mathcal{M}_1 \ar[r] \ar[u]&\mathcal{N}_1 \ar[r] \ar[u] &0 \\
& 0 \ar[u] &0\ar[u] &0\ar[u] & }
\]

\[
\xymatrix{& &0 &0 & \\
&0 \ar[r] & \mathcal{K}_2 \ar[r]^= \ar[u]& \mathcal{K}_2 \ar[u] & \\
0 \ar[r]&\mathcal{O}_X(E_1+E_2) \ar[r] \ar[u] &\mathcal{O}_X(H)\ar[r] \ar[u] &\mathcal{O}_{D_2 + E_2}(H) \ar[r] \ar[u] & 0\\
0 \ar[r]& \mathcal{O}_X(E_1+E_2) \ar[r] \ar[u]^= & \mathcal{M}_2 \ar[r] \ar[u]&\mathcal{N}_2 \ar[r] \ar[u] &0 \\
&0 \ar[u] &0\ar[u] &0\ar[u] & }
\]

show that the support of $N_1$ is $D_1$, and the support of $N_2$ is contained in $D_2\cup E_2$.

$\mathcal{K}_1$ can be supported on a finite number of points or on the whole $D_1$. In the second case we would get that $\mathcal{M}_1$ - the image of $\iota_1$ is equal to $\mathcal{O}_X(E_1 + E_2)$ and hence the sequence
\[
\xymatrix{0 \ar[r] & \mathcal{O}_X(E_1+E_2) \ar[r]^(.7){\psi_1} & V \ar[r]^(.4){\psi_2} & \mathcal{O}_X(E_2) \ar[r] & 0}
\]
splits. But $V$ is a nontrivial extension of $\mathcal{O}_X(E_2)$ by $\mathcal{O}_X(E_1+E_2)$ so we know that $\mathcal{K}_1$ must be supported on a finite number of points. Hence, $c_1(\mathcal{K}_1) = 0$ and $c_1(\mathcal{N}_1) = D_1$.

An analogous argument shows that the support of $\mathcal{K}_2$ can not be equal to $D_2 +E_2$. Hence, $c_1(\mathcal{K}_2)$ is either $E_2$, $D_2$ or $0$ and  $c_1(N_2) = D_2$, $E_2$ or $D_2 + E_2$.

The above diagrams give additional relations
\begin{align*}
H-E_1-E_2 = c_1(\mathcal{N}_i) + c_1(\mathcal{K}_i),\\
0 = c_1(\mathcal{N}_i)c_1(\mathcal{K}_i) + c_2(\mathcal{N}_i) + c_2(\mathcal{K}_i).
\end{align*}

For $\iota_1$ we have:
\begin{align*}
c_1(\mathcal{N}_1) & = D_1,\\
c_1(\mathcal{K}_1) &= c_1(\mathcal{O}_{D_1}(H)) - c_1(\mathcal{N}_1) = H-E_1-E_2 - D_1,\\
c_1(\mathcal{M}_1) &= H - c_1(\mathcal{K}_1) = E_1 + E_2 + D_1,\\
c_1(\mathcal{L}_1) & = E_2 - c_1(\mathcal{N}_1) = E_2 - D_1,\\
c_2(\mathcal{M}_1) & = -c_1(\mathcal{L}_1)c_1(\mathcal{M}_1) = 1,\\
c_2(\mathcal{K}_i) & = -c_1(\mathcal{M}_1)c_1(\mathcal{K}_1) - c_2(\mathcal{M}_1) =
 0,\\
c_2(\mathcal{N}_1) &= c_2(\mathcal{M}_1) - (E_1 + E_2)c_1(\mathcal{N}_1) = 0.
\end{align*}

\begin{itemize}
\item $\mathcal{M}_1 = m_x \otimes \mathcal{O}_X(E_1 + E_2 + D_1)$, where $m_x$ is the maximal ideal of functions vanishing at a point $x\in D_1$, 
\item $\mathcal{L}_1 = \mathcal{O}_X(E_2 - D_1)$,
\item $\mathcal{N}_1 = \mathcal{O}_{D_1}$,
\item $\mathcal{K}_1 = \mathcal{O}_x(H) \simeq \mathcal{O}_x$, (because of the exact sequences $\mathcal{N}_1 \rightarrow \mathcal{O}_{D_1}(H) \rightarrow \mathcal{K}_1$ and $m_x \otimes \mathcal{O}_X(E_1 + E_2 + D_1) \rightarrow \mathcal{O}_X(H) \rightarrow \mathcal{O}_x(H)$).
\end{itemize}

We want to know whether the composition
\[ 
\xymatrix{\mathcal{O}_X \ar[r]^\zeta & V \ar[r]^{\iota_1} & \mathcal{O}_X(H)} 
\]
is equal to 0 or $\eta$.

$\zeta$ and $\iota_1$ fit into a diagram
\[
\xymatrix{& & 0 & &\\
& \mathcal{O}_X \ar[r] &  m_x \otimes \mathcal{O}_X(E_1 + E_2 + D_1)\ar[u] \ar[r]& \mathcal{O}_{\widetilde{D}}&\\
0 \ar[r] & \mathcal{O}_X \ar[r]^\zeta \ar[u]_= & V \ar[r] \ar[u]^{\iota_1} & \mathcal{O}_X(E_1 + 2E_2)\ar[r] \ar[u]& 0\\
&0 \ar[r] \ar[u] & \mathcal{O}_X(E_2 - D_1)\ar[u] \ar[r]^=  & \mathcal{O}_X(E_2 - D_1) \ar[u]&\\
& & 0\ar[u] & &}
\]

We tensor this diagram with $\mathcal{O}_{E_2}$. 

$\mathcal{O}_X(E_1+E_2+D_1)\otimes \mathcal{O}_{E_2} = \mathcal{O}_{E_2}$, $m_x\otimes \mathcal{O}_{E_2} = \mathcal{O}_{E_2}$ as $x \notin E_2$ so there exists an epimorphism $V|_{E_2} \to \mathcal{O}_{E_2}$ with kernel $\mathcal{O}_{E_2}(-1)$. 
\[
\xymatrix{& & 0 & &\\
& & \mathcal{O}_{E_2} \ar[u] & &\\
0 \ar[r] & \mathcal{O}_{E_2} \ar[r]_{\zeta|_{E_2}} & V|_{E_2} \ar[r]_{\psi|_{E_2}} \ar[u]_{\iota_1|_{E_2}} & \mathcal{O}_{E_2}(-1)\ar[r] & 0\\
& & \mathcal{O}_{E_2}(-1)\ar[u]  & &\\
& & 0\ar[u] & &}
\]

It follows that $V|_{E_2} = \mathcal{O}_{E_2}\oplus \mathcal{O}_{E_2}(-1)$ and the composition 
\[ 
\xymatrix{\mathcal{O}_X \ar[r]^{\zeta} &  V \ar[r]^{\iota_1} & \mathcal{M}_1 \ar@{^{(}->}[r] & \mathcal{O}_X(H)}
\] 
restricted to $E_2$ is an isomorphism. Hence $\iota_1 \zeta$ does not have zeros along $E_2$. Changing $\iota_1$ if necessary we obtain that $\iota_1\zeta = \eta$.

For $\iota_2$ there are three possibilities 
$$
c_1(\mathcal{N}_2) = \left\{ \begin{array}{lr} D_2,& \textrm{case }(A)\\ E_2,& \textrm{case }(B) \\ E_2 + D_2. & \textrm{case }(C) \end{array} \right.
$$
Hence,
$$
c_1(\mathcal{M}_2) = \left\{\begin{array}{lr} E_1 + E_2 + D_2,& \textrm{case }(A) \\ E_1 + 2E_2,& \textrm{case }(B) \\ E_1 + 2E_2 + D_2.& \textrm{case }(C) \end{array} \right.
$$
$\mathcal{M}_2$ is a subsheaf of $\mathcal{O}_X(H)$ and hence it is of the form $\mathcal{O}_X(L)\otimes \mathcal{I}_Z$, where $\mathcal{I}_Z$ is an ideal sheaf of a set of points $Z\in D_2 \cup E_2$. Then $c_1(\mathcal{M}_2) = L$ and $c_2(\mathcal{M}_2) = \textrm{deg}(Z) \geq 0$.

In case (A) we have:
$$
\mathcal{L}_2 = \mathcal{O}_X(E_2-D_2), \quad \mathcal{N}_2 = \mathcal{O}_{D_2}, \quad c_1(\mathcal{M}_2) = E_1+E_2+D_2, \quad c_2(\mathcal{M}_2) = -1.
$$
The second Chern class of $\mathcal{M}_2$ is negative and it follows that this case cannot happen.

In case (B) we have:
$$
\mathcal{L}_2 = \mathcal{O}_X, \quad \mathcal{N}_2 = \mathcal{O}_{E_2}, \quad c_1(\mathcal{M}_2) = E_1+2E_2, \quad c_2(\mathcal{M}_2) = 0.
$$
On $X$ we have an inclusion $\mathcal{O}_X(E_1 + E_2) \to \mathcal{M}_2 \to \mathcal{O}_X(H)$. Tensoring with $\mathcal{O}_{D_2}$ we get $\mathcal{O}_{D_2}(E_1 + E_2) = \mathcal{O}_{D_2}(1)$ and $\mathcal{O}_{D_2}(H) = \mathcal{O}_{D_2}(1)$. It follows that  $\mathcal{M}_2\otimes \mathcal{O}_{D_2}$ modulo torsion is equal to $\mathcal{O}_{D_2}(1)$. Whereas, in this case it equals to $\mathcal{O}_{D_2}(2)$.

Thus, we are left we case (C):
$$
\mathcal{L}_2 = \mathcal{O}_X(-D_2), \quad \mathcal{N}_2 = \mathcal{O}_{D_2+E_2}, \quad c_1(\mathcal{M}_2) = E_1+2E_2+D_2, \quad c_2(\mathcal{M}_2) = 1.
$$
We obtain that $\mathcal{M}_2 = m_y\otimes \mathcal{O}_X(E_1+2E_2+D_2)$ for some $y\in E_2 \cup D_2$. If $y \in D_2$ then $\mathcal{M}_2\otimes \mathcal{O}_{D_2} = \mathcal{O}_{D_2}$ and there is no epimorphism from $V|_{D_2} = \mathcal{O}_{D_2}(1) \oplus \mathcal{O}_{D_2}(1)$ onto $\mathcal{O}_{D_2}$. So $y\in E_2 \setminus D_2$.

Thus, we know that $\iota_2\zeta$ is zero on a point $y\in E_2$. Hence it has zeros along $E_2$ and one can choose $\iota_2$ in such a way that $\iota_2 \zeta =0$.

To sum up, the collection (not exceptional!) \mbox{$\langle \mathcal{O}_X, V, \mathcal{O}_X(E_1 +E_2), \mathcal{O}_X(H), \mathcal{O}_X(2H) \rangle$} has a quiver 
\[
\xymatrix{\mathcal{O}_X \ar[r]^{\zeta} & V \ar[r]^{\beta\phi_2} \ar@/_2pc/[rr]^{\iota_1} \ar@/_2pc/@<-2ex>[rr]_{\iota_2}& \mathcal{O}_X(E_1+E_2) \ar@<1ex>[l]^{\phi_1}  & \mathcal{O}_X(H) \ar@<2ex>[r]^{\delta_1} \ar[r]^{\delta_2} \ar@<-2ex>[r]^{\delta_3} & \mathcal{O}_X(2H)}
\]
with relations:
\begin{align*}
&\iota_1\zeta = \eta,  \quad \beta\phi_2\phi_1 =0,  \quad \iota_2\zeta = 0, \\  &\delta_1\iota_1\zeta = \delta_3\iota_1\phi_1\beta\phi_2\zeta,  \quad \delta_2\iota_1\zeta  = \delta_3\iota_2\phi_1\beta\phi_2\zeta,\\
&\delta_1\iota_2 - \delta_2\iota_1  \in \textrm{span}\{\delta_1 \iota_1 \phi_1 \beta \phi_2, \delta_1 \iota_2 \phi_1 \beta \phi_2, \delta_2 \iota_2 \phi_1 \beta \phi_2, \delta_3 \iota_1 \phi_1 \beta \phi_2, \delta_3 \iota_2 \phi_1 \beta \phi_2\}.
\end{align*}

Recall that some of the morphisms can be changed:
\begin{align*}
\gamma_1 & \rightsquigarrow \gamma_1 + a\gamma_2\\
\eta &\rightsquigarrow \eta + b\gamma_1\beta\alpha + c\gamma_2\beta \alpha,\\
\zeta & \rightsquigarrow \zeta + d\phi_1\beta \alpha,\\
\iota_1 & \rightsquigarrow \iota_1 + e\gamma_1\beta \phi_2 + f\gamma_2\beta \phi_2,\\
\iota_2 & \rightsquigarrow \iota_2 + g\gamma_1\beta \phi_2 + h\gamma_2\beta \phi_2
\end{align*}
and $\delta_1 \leftrightarrow \gamma_1\beta\alpha$, $\delta_2\leftrightarrow \gamma_2\beta\alpha$, $\delta_3\leftrightarrow \eta$.

A change $\gamma_1 \rightsquigarrow \gamma_1 + a\gamma_2$ changes $\iota_1$ to $\iota_1 + a\iota_2$ (because $\iota_1 \phi_1 = \gamma_1$). As also $\delta_1$ depends on $\gamma_1$ we get
$$
\delta_1\iota_2 - \delta_2\iota_1 \rightsquigarrow (\delta_1 + a\delta_2)\iota_2 - \delta_2(\iota_1 + a\iota_2) = \delta_1\iota_2 + a\delta_2\iota_2 - \delta_2\iota_1 - a \delta_2\iota_2 = \delta_1\iota_2 - \delta_2\iota_1
$$
and hence the parameter $a$ has no influence on the relation between $\delta_1\iota_2$ and $\delta_2 \iota_1$.

As we want $\iota_2\zeta = 0$ the calculations
$$
(\iota_2 + g\gamma_1\beta \phi_2 + h\gamma_2\beta \phi_2)(\zeta + d\phi_1\beta \alpha) = \iota_2 \zeta + (d+h) \gamma_2\beta\alpha + g \gamma_1\beta\alpha
$$
show that $g$ must be $0$ and $h$ can be arbitrary, as putting $d=-h$ will not lead to any changes in this relation. 

The morphism $\eta = \iota_1\zeta$ can be changed by any combination of $\gamma_1\beta \alpha$ and $\gamma_2\beta \alpha$ so there is no condition on $e$ and $f$.

Calculating
$$
\delta_1(\iota_2 + h\gamma_2\beta\phi_2)- \delta_2(\iota_1 + e\gamma_1\beta \phi_2 + f\gamma_2\beta \phi_2) = \delta_1\iota_2 - \delta_2\iota_1 + (h-e) \delta_1\gamma_2\beta\phi_2 - f \delta_2\gamma_2\beta\phi_2
$$
we see that $e$ and $f$ can be chosen in such a way that
$$
\delta_1\iota_2 - \delta_2\iota_1  = A \delta_1 \iota_1 \phi_1 \beta \phi_2 + B \delta_3 \iota_1 \phi_1 \beta \phi_2+ C \delta_3 \iota_2\phi_1 \beta \phi_2.
$$
Furthermore, the relation $\delta_3\iota_2\phi_1\beta\phi_2\zeta = \delta_2\iota_1\zeta$ gives
\begin{align*}
0  = \delta_3\iota_2\phi_1\beta\phi_2\zeta - \delta_2\iota_1\zeta &= \delta_3\iota_2\phi_1\beta\phi_2\zeta - \delta_1\iota_2\zeta + A \delta_1 \iota_1 \phi_1 \beta \phi_2\zeta + B \delta_3 \iota_1 \phi_1 \beta \phi_2\zeta \\ + C \delta_3 \iota_2\phi_1 \beta \phi_2\zeta & = A \delta_1 \iota_1 \phi_1 \beta \phi_2\zeta + B \delta_3 \iota_1 \phi_1 \beta \phi_2\zeta + (C+1) \delta_3 \iota_2\phi_1 \beta \phi_2\zeta
\end{align*}
and leads to $A = 0 =B$ and $C=-1$ due to linear independence of $\delta_1 \iota_1 \phi_1 \beta \phi_2\zeta$, $\delta_3 \iota_1 \phi_1 \beta \phi_2\zeta$ and $ \delta_3 \iota_2\phi_1 \beta \phi_2\zeta$. 

Note, that the relation $\delta_1\iota_2 + \delta_3\iota_2\phi_1\beta\phi_2 = \delta_2\iota_1$ composed with $\zeta$ gives $\delta_3\iota_2\phi_1\beta\phi_2\zeta = \delta_2\iota_1\zeta$.

Hence, relations in the above quiver are:
\begin{align*}
\iota_2\zeta &=0, & \beta\phi_2\phi_1 & = 0,&\\
\delta_1\iota_1\zeta & = \delta_3\iota_1\phi_1\beta\phi_2\zeta, & \delta_1\iota_2 + \delta_3\iota_2\phi_1\beta\phi_2 & = \delta_2\iota_1.&
\end{align*}

The collection $\langle \mathcal{O}_X, \mathcal{O}_X(E_2), \mathcal{O}_X(E_1+E_2), \mathcal{O}_X(H), \mathcal{O}_X(2H) \rangle$ has the DG quiver of the collection
$$
\begin{array}{ccccc} & \mathcal{O}_X(E_1 +E_2) & & & \\ (\mathcal{O}_X, &\Big\downarrow\rlap{$\scriptstyle{\phi_1}$} & \mathcal{O}_X(E_1+E_2), & \mathcal{O}_X(H), & \mathcal{O}_X(2H)) \\& V & & &  \end{array}
$$

Morphisms from $\mathcal{O}_X$ to $\mathcal{O}_X(E_2)$ have basis
\[
\xymatrix{& & \mathcal{O}_X(E_1+E_2)\ar[dd]^{\phi_1} & & &\mathcal{O}_X(E_1+E_2) \ar[dd]^{\phi_1} \\
a_1 = & &  & a_2 = & & \\
         & \mathcal{O}_X\ar[uur]^{\beta\phi_2\zeta}& V & & \mathcal{O}_X\ar[r]^{\beta\phi_2\zeta\phi_1}& V}
\]
\[
\xymatrix{ & & & \mathcal{O}_X(E_1 + E_2)\ar[dd]^{\phi_1} \\
 &  a_3 = & &\\
& & \mathcal{O}_X \ar[r]^\zeta & V}
\]
with $\partial(a_1) = a_2$.

Morphisms from $\mathcal{O}_X(E_2)$ to $\mathcal{O}_X(E_1+E_2)$ have basis:
\[
\xymatrix{ &\mathcal{O}_X(E_1+E_2)\ar[dd]^{\phi_1} \ar[ddr]^{\textrm{id}} & & & \mathcal{O}_X(E_1+E_2) \ar[dd]^{\phi_1} &  \\
b_1 = & & & b_2 = & &          \\
      & V & \mathcal{O}_X(E_1+E_2)& &V \ar[r]^{\beta\phi_2}& \mathcal{O}_X(E_1+E_2)      \\}
\]
with $b_1$ in degree 1.

Morphisms from $\mathcal{O}_X(E_2)$ to $\mathcal{O}_X(H)$ are
\[
\xymatrix{& \mathcal{O}_X(E_1+E_2)\ar[dd]^{\phi_1} & & & \mathcal{O}_X(E_1+E_2) \ar[dd]^{\phi_1} \ar[ddr]^{\iota_1\phi_1} & & &  \\
c_1 = & & &  c_2 = & & &    \\
&  V \ar[r]^{\iota_1}& \mathcal{O}_X(H) & & V& \mathcal{O}_X(H) & }
\]
\[
\xymatrix{ & \mathcal{O}_X(E_1+E_2) \ar[dd]^{\phi_1} & & & \mathcal{O}_X(E_1+E_2)\ar[dd]^{\phi_1} & & \\
 c_3= & & & c_4 = & & & \\
& V \ar[r]^{\iota_1\phi_1\beta\phi_2} & \mathcal{O}_X(H) & &  V \ar[r]^{\iota_2}& \mathcal{O}_X(H) & }
\]
\[
\xymatrix{ & \mathcal{O}_X(E_1+E_2) \ar[dd]^{\phi_1} \ar[ddr]^{\iota_2\phi_1} & & & \mathcal{O}_X(E_1+E_2) \ar[dd]^{\phi_1} & & \\
  c_5 = & & & c_6= & & &     \\
& V& \mathcal{O}_X(H) & & V \ar[r]^{\iota_2\phi_1\beta\phi_2} & \mathcal{O}_X(H) &        \\}
\]
with $\partial(c_1) = c_2$ and $\partial(c_4) = c_5$.

Morphisms from $\mathcal{O}_X(E_2)$ to $\mathcal{O}_X(2H)$ factor through $\mathcal{O}_X(H)$ so can be written as $\delta_i c_j$.

Relations between these morphisms are
\begin{align*}
b_1 a_3 & = 0,& b_2a_3 & = \beta \phi_2 \zeta ,& \iota_1 \phi_1 b_1 & = c_2,&\\
\iota_1 \phi_1 b_2 & = c_3,& \iota_2 \phi_1 b_1 & = c_5,& \iota_2 \phi_1 b_2 & = c_6,&\\
\partial(c_1) & = \iota_1 \phi_1 b_1,& \partial(c_4) & = \iota_2 \phi_1 b_1,& c_1 a_3 & = \iota_1 \zeta,& \\
c_4 a_3 & =0,& \delta_2 c_1 a_3 & = \delta_3 \iota_2 \phi_1 b_2 a_3, & \delta_1 c_4 & = \delta_2 c_1,&\\
\delta_1 c_1 a_3 & = \delta_3 \iota_1 \phi_1 b_2 a_3,& \delta_1c_4 + \delta_3c_6  & = \delta_2 c_1.& & &
\end{align*}

Putting $\alpha = a_3$, $\beta= b_2$, $\bar{\beta} = b_1$, $\gamma_1 = \iota_1 \phi_1$, $\gamma_2 = \iota_2 \phi_2$, $\epsilon_1 = c_1$ and $\epsilon_2 = c_4$ we obtain the following DG quiver 
\[
\xymatrix{\mathcal{O}_X \ar[r]^{\alpha}  & \mathcal{O}_X(E_2) \ar[r]^{\beta} \ar@/^1pc/[r]^{\bar{\beta}} \ar@/_2pc/[rr]^{\epsilon_1} \ar@/_2pc/@<-2ex>[rr]_{\epsilon_2} & \mathcal{O}_X(E_1+E_2) \ar[r]^{\gamma_1} \ar@<-1ex>[r]_{\gamma_2} & \mathcal{O}_X(H) \ar@<2ex>[r]^{\delta_1} \ar[r]^{\delta_2} \ar@<-2ex>[r]^{\delta_3}& \mathcal{O}_X(2H)}
\]
with relations and differentials given by
\begin{align*}
\partial(\epsilon_1) & = \gamma_1 \bar{\beta}, & \partial(\epsilon_2) & = \gamma_2 \bar{\beta},& \epsilon_2 \alpha & = 0,&\\
\delta_1 \gamma_2 & = \delta_2 \gamma_1, &  \delta_2 \epsilon_1\alpha & = \delta_3 \gamma_2 \beta \alpha,& \delta_1\epsilon_2 + \delta_3 \gamma_2 \beta &= \delta_2\epsilon_1.& 
\end{align*}

\begin{rem}
In this case, in order to find the DG category associated to an exceptional collection one has to calculate triple Massey product of morphisms between its elements. 
\end{rem}

\section{Non-commutative deformation}

As in the above example let $X$ be the blow up of $\mathbb{P}^2$ along the degree 2 subscheme supported at a point $x_0$ and let $Y$ be the blow up of $\mathbb{P}^2$ along two different points $x_0$ and $x_1$. 

The Picard group of $X$ is generated by divisors $E_1$, $E_2$ and $H_X$ while the Picard group of $Y$ -- by divisors $F_1$, $F_2$ and $H_Y$. 

Let $\mathcal{X} \to \mathbb{P}^1$ be the deformation space of $Y$ to $X$. There exists divisors $D_0$, $D_1$ and $D_2$ on $\mathcal{X}$ such that 
\begin{align*}
D_0|_{\{t=0\}} & = H_X, & D_0|_{\{t \neq 0\}} &= H_Y,\\
D_1|_{\{t=0\}} & = E_1 + E_2, & D_1|_{\{t\neq 0\}} &= F_1,\\
D_2|_{\{t=0\}} & = E_2, & D_2|_{\{t\neq 0\}} &= F_2.
\end{align*}

Thus, there exists a collection $\sigma = \langle \mathcal{L}_1, \ldots,\mathcal{L}_5 \rangle$ on $\mathcal{X}$ such that 
\begin{align*}
\sigma|_{\{t=0\}} & = \langle \mathcal{O}_X, \mathcal{O}_X(E_2), \mathcal{O}_X(E_1+E_2), \mathcal{O}_X(H_X), \mathcal{O}_X(2H_X) \rangle,\\
\sigma|_{\{t\neq 0\}} &= \langle \mathcal{O}_Y, \mathcal{O}_Y(F_2), \mathcal{O}_Y(F_1), \mathcal{O}_Y(H_Y), \mathcal{O}_Y(2H_Y)\rangle
\end{align*} 

\begin{thm}\label{thm:deformation}
There exists a non-commutative deformation of $Y$to $X$ -- a deformation of a DG category of the collection \mbox{$\langle \mathcal{O}_Y, \mathcal{O}_Y(F_2), \mathcal{O}_Y(F_1), \mathcal{O}_Y(H_Y), \mathcal{O}_Y(2H_Y) \rangle$} to a DG category of the collection \mbox{$\langle \mathcal{O}_X, \mathcal{O}_X(E_2), \mathcal{O}_X(E_1+E_2), \mathcal{O}_X(H_X), \mathcal{O}_X(2H_X) \rangle.$} 
\end{thm}

\begin{proof}
Let $t\in \mathbb{P}^1$ be a parameter and let us consider a DG quiver $\Delta(t)$ 
\[
\xymatrix{&\bullet \ar@<-1ex>[dd]_{\beta} \ar[dd]^{\bar{\beta}} \ar[dr]_{\epsilon_1} \ar@<1ex>[dr]^{\epsilon_2} & & \\
\bullet \ar[ur]^{\alpha} & &\bullet \ar@<2ex>[r]^{\delta_1} \ar[r]^{\delta_2} \ar@<-2ex>[r]^{\delta_3} & \bullet  \\
& \bullet \ar@<-1ex>[ur]^{\gamma_1} \ar@<-2ex>[ur]_{\gamma_2} & &} 
\]
with $\bar{\beta}$ in degree 1, differentials
$$
\partial{\beta}  = t \bar{\beta}, \quad \partial(\epsilon_1)  = \gamma_1\bar{\beta},\quad \partial(\epsilon_2)  = \gamma_2\bar{\beta}
$$

and relations
\begin{align*}
\epsilon_2 \alpha& = 0, & \bar{\beta}\alpha & = 0,& \delta_1\epsilon_1\alpha & = \delta_3\gamma_1\beta\alpha, &\\
\delta_1\gamma_2 & = \delta_2\gamma_1, & \delta_1\epsilon_2 + \delta_3\gamma_2 \beta & = \delta_2\epsilon_1 + t \delta_3\epsilon_2.&&&
\end{align*}

Let $D(t)$ denote the DG algebra of paths of this DG quiver.

For $t=0$, $\Delta(t)$ is the DG quiver of the collection $\langle \mathcal{O}_X, \mathcal{O}_X(E_2), \mathcal{O}_X(E_1+E_2), \mathcal{O}_X(H_X), \mathcal{O}_X(2H_X) \rangle$. 

For $t\neq 0$, $D(t)$ has cohomology groups concentrated in degree 0 and hence is quasi-isomorphic to the path algebra of the following quiver
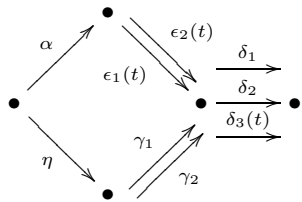
\[
\xymatrix{&\bullet \ar[dr]_{\epsilon_1(t)} \ar@<1ex>[dr]^{\epsilon_2(t)} & & \\
\bullet \ar[dr]_{\eta} \ar[ur]^{\alpha} & &\bullet \ar@<3ex>[r]^{\delta_1} \ar[r]^{\delta_2} \ar@<-3ex>[r]^{\delta_3(t)} & \bullet  \\
& \bullet \ar@<-1ex>[ur]^{\gamma_1} \ar@<-2ex>[ur]_{\gamma_2} & &} 
\]
where
\begin{align*}
\eta & = \beta\alpha, & \epsilon_1(t) & = \gamma_1\beta - t\epsilon_1, \\
\epsilon_2(t) & = \gamma_2\beta - t \epsilon_2, & \delta_3(t) & = \delta_1 - t\delta_3.
\end{align*}

Relations between morphisms are:
\begin{align*}
\delta_1\gamma_2 & = \delta_2\gamma_1, \\
\epsilon_2(t)\alpha & = \gamma_2\beta\alpha - t\epsilon_2\alpha = \gamma_2\eta,\\
\delta_2\epsilon_1(t) & = \delta_2\gamma_1\beta - t\delta_2\epsilon_1 = \delta_1\gamma_2\beta -t(\delta_1\epsilon_2 + \delta_3\gamma_2\beta - t\delta_3\epsilon_2) \\
& = (\delta_1-t\delta_3)(\gamma_2\beta - t\epsilon_2) = \delta_3(t)\epsilon_2(t).
\end{align*}
As $\epsilon_1 = \frac{1}{t}(\gamma_1\beta -\epsilon_1(t))$ and $\delta_3 = \frac{1}{t}(\delta_1-\delta_3(t))$ the relation $\delta_1\epsilon_1\alpha = \delta_3\gamma_1\beta\alpha $ can be written as:
\begin{align*}
0 & = \delta_1(\gamma_1\beta - \epsilon_1(t))\alpha -(\delta_1 - \delta_3(t))\gamma_1\eta \\
& = \delta_1\gamma_1\eta - \delta_1\epsilon_1(t)\alpha - \delta_1\gamma_1\eta + \delta_3(t)\gamma_1\eta = \delta_3(t)\gamma_1\eta - \delta_1\epsilon_1(t)\alpha. 
\end{align*}
Hence, for $t\neq 0$ the quiver $\Delta(t)$ is quasi-isomorphic (its DG path algebra is quasi-isomorphic) to (the path algebra of) the quiver
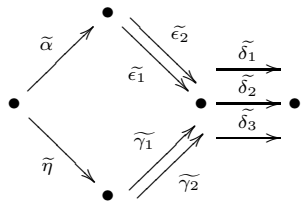
\[
\xymatrix{&\bullet \ar[dr]_{\widetilde{\epsilon_1}} \ar@<1ex>[dr]^{\widetilde{\epsilon_2}} & & \\
\bullet \ar[dr]_{\widetilde{\eta}} \ar[ur]^{\widetilde{\alpha}} & &\bullet \ar@<3ex>[r]^{\widetilde{\delta_1}} \ar[r]^{\widetilde{\delta_2}} \ar@<-3ex>[r]^{\widetilde{\delta_3}} & \bullet  \\
& \bullet \ar@<-1ex>[ur]^{\widetilde{\gamma_1}} \ar@<-2ex>[ur]_{\widetilde{\gamma_2}} & &} 
\]
with relations
\begin{align*}
\widetilde{\delta_1}\widetilde{\gamma_2} & = \widetilde{\delta_2}\widetilde{\gamma_1}, &
\widetilde{\delta_3}\widetilde{\epsilon_2} & = \widetilde{\delta_2}\widetilde{\epsilon_1},& \\ \widetilde{\epsilon_2}\widetilde{\alpha} & = \widetilde{\gamma_2}\widetilde{\eta}, &\widetilde{\delta_3}\widetilde{\gamma_1}\widetilde{\eta} & = \widetilde{\delta_1}\widetilde{\epsilon_1}\widetilde{\alpha}.&
\end{align*}

This is the quiver with relations of a collection $\langle \mathcal{O}_Y, \mathcal{O}_Y(F_2), \mathcal{O}_Y(F_1), \mathcal{O}_Y(H_Y), \mathcal{O}_Y(2H_Y) \rangle$ on $Y$.
\end{proof}
Note that if we denote by $L_i = Z(\widetilde{\delta_i})$ then $x_0 \in L_1$, $x_1\in L_3$ and both $x_0$ and $x_1$ lie on $L_2$. The described deformation does not influence neither $L_1$, $L_2$ nor $x_0$. For $t\neq 0$, let $L_3(t) = Z(\delta_3(t))$ and $x_1(t) = L_1\cap L_3(t)$. Then as $t$ goes to $0$ the point $x_1(t)$ goes to $x_0$ along the fixed line $L_2$ and $L_3(t)$ is approaching $L_1$. 

\vspace{0.6cm}

\textbf{Acknowledgements} The author is greatly indebted to prof. Alexey Bondal for the introduction to the problem and many fruitful discussions and to prof. Adrian Langer for drawing the author's attention the paper \cite{bib_HP} and a lot of help with the computations.

\end{document}